\documentclass[final,leqno,onefignum,onetabnum]{article}

\usepackage[square,numbers,sort]{natbib}

\usepackage{url}            
\usepackage{booktabs}       
\usepackage{amsmath,amsthm,amssymb,amsfonts}       
\usepackage{nicefrac}       
\usepackage{microtype}      
\usepackage{dsfont}
\usepackage{latexsym}
\usepackage{lipsum}
\usepackage{graphicx}
\usepackage{epstopdf}
\usepackage{algorithmic}
\usepackage{hyperref}
\usepackage{cleveref}
\usepackage{autonum}
\usepackage{color}

\newcommand{\esp}{\mathbb{E}}
\newcommand{\prob}{\mathbb{P}}

\newcommand{\re}{\mathbb{R}}

\DeclareMathOperator*{\RE}{RE}

\def\bfTheta{\boldsymbol{\Theta}}
\def\bfSigma{\boldsymbol{\Sigma}}
\def\bfA{\mathbf{A}}
\def\bfB{\mathbf{B}}
\def\bfE{\mathbf{E}}
\def\bfI{\mathbf{I}}
\def\bfM{\mathbf{M}}
\def\bfX{\mathbf{X}}
\def\bfY{\mathbf{Y}}

\def\bb{\boldsymbol b}

\def\by{\boldsymbol y}
\def\bu{\boldsymbol u}
\def\bv{\boldsymbol v}
\def\bw{\boldsymbol w}

\def\bmu{\boldsymbol\mu}
\def\btheta{\boldsymbol\theta}

\def\calC{\mathcal C}
\def\calN{\mathcal N}

\newcommand{\vertiii}[1]{{\left\vert\kern-0.4ex #1 
\kern-0.4ex\right\vert}}

\theoremstyle{plain}
\newtheorem{definition}{Definition}
\newtheorem{lemma}{Lemma}
\newtheorem{thm}{Theorem}
\newtheorem{cor}{Corollary}

\newtheorem{assump}{Assumption}
\theoremstyle{definition}

\newtheorem{example}{Example}

\usepackage{geometry}
\begin{document}

\title{Restricted eigenvalue property for corrupted Gaussian designs}

\author{Philip Thompson\footnote{Philip Thompson is supported by a Fondation Math\'ematique Jacques Hadamard post-doctoral fellowship. This work was supported by a public grant as part of the Investissement d'avenir project, reference ANR-11-LABX-0056-LMH, LabEx LMH.} \quad and \quad Arnak S. Dalalyan\\
 Centre de Recherche en \'Economie et Statistique (CREST-ENSAE) \\ 
 \href{}{Philip.THOMPSON@ensae.fr}\\
 \href{}{arnak.dalalyan@ensae.fr}
}
\date{}
\maketitle

\begin{abstract}
Motivated by the construction of tractable robust estimators via convex relaxations, we present conditions on the sample size which guarantee an augmented notion of Restricted Eigenvalue-type condition for Gaussian designs. Such a notion is suitable for 
high-dimensional robust inference in a Gaussian linear model and a multivariate 
Gaussian model when samples are corrupted by outliers either in the response variable or in the design matrix. Our proof technique relies on simultaneous lower and upper bounds of two random bilinear forms with very different behaviors. Such simultaneous bounds are used for balancing the interaction between the parameter vector and the estimated corruption vector as well as for controlling the presence of corruption in the design. Our technique has the advantage of not relying on known bounds of the extreme singular values of the associated Gaussian ensemble nor on the use of mutual incoherence arguments. A relevant consequence of our analysis, compared to prior work, is that a significantly sharper restricted eigenvalue constant can be obtained under weaker assumptions. In particular, the sparsity of the unknown parameter and the number of outliers are allowed to be completely independent of each other. 
\end{abstract}

\section{Introduction}

As it is widely known, high-dimensional inference problems suffer from the curse of dimensionality in the sense that the sample size $n$ is much smaller than the dimension of the parameter (for which ones wishes to estimate according to certain risk measure). However, under sparsity assumptions, a celebrated methodology is to enforce variable selection over a lower dimensional subspace by using convex relaxations. Celebrated examples are the Lasso and Dantzig estimators via the $\ell_1$-norm \citep{tibshirani1996,candes:tao2007} and matrix estimation problems using the nuclear norm 
\citep{candes:recht2009}. A highlight of the convex relaxation approach is that the corresponding estimator can be efficiently computed even for large-scale optimization problems. It should be noted, however, that certain assumptions on the data must be met in order for such approaches to work. 

Perhaps the most common assumption for sparse recovery is 
the \emph{restricted eigenvalue} (RE) condition \citep{BRT}. Given a set 
${S}\subset[p]:=\{1,\ldots,p\}$ and a constant $c>0$, we first 
define the dimension reduction cone 
$$
{\calC}_{{S}}(c):=\big\{\bb\in\re^p:\Vert \bb_{{S}^c}\Vert_1\le c\Vert 
\bb_{{S}}\Vert_1\big\}.
$$
Hereafter $\bb_{S}$ is the vector obtained from $\bb$ by zeroing its coordinates 
$i\notin{S}$. Let $\bfX\in\re^{n\times p}$ be some matrix. We say the 
$\RE_{{S}}(c)$ property holds for $\bfX$ with constant $\kappa>0$ if 
\begin{align}
\Vert\bfX \bv\Vert_{2}\ge\kappa\Vert \bv\Vert_2,
\quad\quad
\forall \bv\in{\calC}_{{S}}(c),
\label{equation:RE:standard}
\end{align}
If the above property holds for all ${S}\subset[p]$ of size 
$|{S}|=s$, we say the $\RE_{s}(c)$ property holds for $\bfX$ 
with constant $\kappa>0$. 
It turns out that such a property is satisfied with high-probability for many
families of random matrices $\bfX$, for example when the rows of $\bfX$ are 
independent and identically distributed Gaussian vectors 
\citep{candes2008,raskutti:wainwright:yu2010}. 

On the other hand, in statistics and machine learning we are often faced with the presence of \emph{outliers} in data, a framework pertaining to the field of 
\emph{robust statistics} \citep{huber:ronchetti2009}. In linear regression, for
instance, outliers can affect not only the labels \citep{nguyen:tran2013} but also the features \citep{balmand:dalalyan2015}. The outliers can be modeled as deterministic 
or random, drawn from some distribution. Of course, the latter is not known in practice. 

Taking into consideration the aforementioned features and having in mind the high-dimensional setting, a recent successful approach consists in modeling
the outliers using a sparse array (vector or matrix). This means that an array characterizing the corruption by the outlier is introduced, which has $|O|$-sparse, 
where $O$ is the set of outliers and $|O|$ is its cardinality. We refer, e.g., to 
\citep{Loh2018,candes:randall2008,dalalyan:keriven2009,dalalyan:keriven2012,dalalyan:chen2012,nguyen:tran2013,tarr:muller:weber2016,balmand:dalalyan2015}. In some of these works, the constructed estimators use the convex relaxation methodology for designing computationally tractable robust estimators. 
Tight risk bounds for these estimators are obtained under the assumption that
a version of RE condition, hereafter referred to as \emph{augmented RE condition}
(see \Cref{ss:definitions} for the precise definition), holds. A fundamental question is how to ensure that the \emph{augmented RE condition} is valid with high probability 
for random design matrices. 

The main purpose of this work is to obtain such guarantees for a broad class of corrupted Gaussian models studied in prior work (e.g. \citep{dalalyan:chen2012,nguyen:tran2013,balmand:dalalyan2015}). Besides contamination, either in the response variable or the design matrix, we are interested in ensuring the RE property in the typical statistical framework where the covariates can be correlated and the design is not under control of the statistician. As it will be specified later (see Section \ref{s:results}), our guarantees for the RE condition are stronger than given in previous work 
\citep{nguyen:tran2013} in the context of linear regression models with corruption only on the response variable. In a nutshell, our improvements in this set-up are as follows: (1) we do not assume any relation between the sparsity level of the parameter and the corruption and (2) our RE constant is significantly sharper (in at least one order of magnitude) and valid for smaller sample sizes. Moreover, we are not aware of previous results ensuring the augmented RE property for models with contamination in the design matrix \citep{balmand:dalalyan2015}. Our results are also relevant for multivariate regression problems where one wishes to estimate a matrix parameter. In this setting, a proper RE condition over the space of matrices needs to be guaranteed, a result which seems to be new. We refer to the following Examples \ref{example:robust:lasso}-\ref{example:robust:precision} and to Remarks \ref{rem:robust:lasso}-\ref{rem:robust:precision:matrix} in Section \ref{s:results} for a more detailed discussion. 

\section{Notations and definitions}
\label{ss:definitions}

We first present some additional notations. Given $\bb\in\re^p$ and a subset 
$S\subset[p]:=\{1,\ldots,p\}$, $b_{S}$ is the vector obtained from $\bb$ 
by zeroing the coordinates in $[p]\setminus{S}$. As usual, for $q\in(0,\infty]$, $\Vert \bb\Vert_q$ will denote the $\ell_q$-norm of $\bb$. For every matrix $\bfA \in\re^{n\times p}$, we shall denote by $\bfA _{\bullet,j}$ the $j$-th column of $\bfA $ and by $\bfA _{i,\bullet}$ its $i$-th row. More generally, if $O$ is a subset of $[n]=\{1,\ldots,n\}$, then $\bfA _{O,\bullet}\in\re^{|O|\times p}$ will denote the matrix defined by suppressing from $\bfA$ the rows corresponding to indices $i\in[n]\setminus O$. 
A similar notation is used for the columns. For two matrices $\bfB$ and $\bfTheta$
having the same number of columns, we write $[\bfB;\bfTheta]$ for the matrix
obtained by vertical concatenation of $\bfB$ and $\bfTheta$.

Let $\mathcal{J}=\{J_j:j\in[p]\}$ be a collection of $p$ subsets $J_j\subset[p]$ 
and let $|\mathcal{J}|:=\sum_{j=1}^p|J_j|$. With a slight abuse of notation, let 
us write $\mathcal{J}^c$ for the collection $\{J_j^c:j\in[p]\}$. Then, for every 
matrix $\bfA \in\re^{p\times p}$, we define $\bfA _{\mathcal{J}}\in\re^{p\times p}$
as the matrix obtained from $\bfA $ by zeroing all elements $\bfA _{i,j}$ such that 
$i\notin J_j$. For matrix mixed-norms, we shall use the following notation: for any
matrix $\bfA \in\re^{n\times p}$ and $q_1,q_2\in(0,\infty]$,
$$
\Vert\bfA \Vert_{q_1,q_2}:=\Big({\sum_{i=1}^n\Vert\bfA _{i,\bullet}\Vert_{q_1}^{q_2}}\Big)^{1/q_2}.
$$ 
The smallest and the largest singular values of $\bfA $ will be denoted by 
$\sigma_{\min}(\bfA )$ and $\sigma_{\max}(\bfA):=\|\bfA\|$, respectively. To ease 
notation, we also define the normalized matrix $\bfA^{(n)}:=
{\bfA }/{\sqrt{n}}$ for any matrix $\bfA$ having $n$ rows. Finally, $\bfA^{\dagger}$ 
will denote the pseudo-inverse of the matrix $\bfA$ and $\bfA^{\dagger/2}$ is
the pseudo-inverse of the matrix $\bfA^{1/2}$ (for a positive semidefinite 
matrix $\bfA$).

We present now the definitions of the (augmented-) dimension-reduction cone and 
the augmented restricted eigenvalue property by starting by the vector case and then
extending these notions to matrices.

\begin{definition}[Augmented RE-condition for vectors]\label{def:RE:vectors}
Let  $\gamma>0$, $c>1$, ${S}\subset[p]$ and $O\subset[n]$. We define the 
\emph{augmented dimension reduction cone} ${\calC}_{{S},O}
(c,\gamma)$ as the set of all vectors $\begin{bmatrix}\bb;\btheta\end{bmatrix}
\in\re^{p+n}$ satisfying
\begin{align}
\gamma\Vert\bb_{{S}^c}\Vert_{1}+\Vert\btheta_{O^c}\Vert_{1}
\le c \left(\gamma\Vert \bb_{{S}}\Vert_{1}+\Vert\btheta_{O}\Vert_1\right).
\end{align}
We say that the $n\times p$ matrix $\bfM$ satisfies the augmented RE condition
$\RE_{{S},O}(c,\gamma)$ with constant $\kappa>0$ if
\begin{align}
\Vert\bfM \bb-\btheta\Vert_{2}\ge\kappa\big\|\begin{bmatrix}\bb;\btheta\end{bmatrix}
\big\|_2,
\quad 
\forall \begin{bmatrix}\bb;\btheta\end{bmatrix}
\in{\calC}_{{S},O}(c,\gamma).
\end{align}
If the above property is holds for all ${S}\subset[p]$ and $O\subset[n]$ 
of size $|{S}|=s$ and $|O|=o$, we say that $\bfM$ satisfies the 
$\RE_{s,o}(c,\gamma)$ condition with constant $\kappa>0$.
\end{definition}

By replacing $\btheta$ by zero, one easily checks that the augmented
RE is stronger than the standard RE condition. Since the Gaussian matrices
are known to satisfy the standard RE condition, it is an appealing question
to check whether they satisfy the (stronger) augmented RE condition.

We now present the definition of the augmented RE condition over 
$[\bfB;{\bfTheta}]\in\re^{(p+n)\times p}$. 

\begin{definition}[Augmented RE-condition for matrices]\label{def:RE}
Let $c>1$, $\gamma>0$, $O\subset[n]$ and $\mathcal{J}=\{J_j:j\in[p]\}$ 
be a collection of $p$ subsets $J_j\subset[p]$.  The \emph{augmented dimension 
reduction cone} ${\calC}_{\mathcal{J},O}(c,\gamma)$ 
is the set of all matrices  $[\bfB;\bfTheta]\in\re^{(p+n)\times p}$  such that
\begin{align}
\gamma\Vert\bfB_{\mathcal{J}^c}\Vert_{1,1}
+\Vert{\bfTheta}_{O^c,\bullet}\Vert_{2,1}
\le c \Big\{\gamma\Vert\bfB_{\mathcal{J}}\Vert_{1,1}
+\Vert{\bfTheta}_{O,\bullet}\Vert_{2,1}\Big\}.
\end{align}
We say that an $n\times p$ matrix $\bfM$ satisfies the 
$\RE_{\mathcal{J},O}(c,\gamma)$ condition with constant 
$\kappa>0$ if
\begin{align}
\Vert\bfM\bfB-\bfTheta\Vert_{2,2}\ge\kappa\left(\left\Vert\bfB\right\Vert_{2,2}\vee\left\Vert{\bfTheta}\right\Vert_{2,2}\right),
\end{align}
for all $[\bfB;\bfTheta]\in{\calC}_{\mathcal{J},O}(c,\gamma)$. 
If the above property holds over all supports 
$\mathcal{J}=\{J_j:j\in[p], J_j\subset[p]\}$ of $[p]\times[p]$ and $O\subset[n]$ with sizes $|\mathcal{J}|:=\sum_j |J_j|=s$ and $|O|=o$, we say that $\bfM$ satisfies the 
$\RE_{s,o}^{\rm mat}(c,\gamma)$ condition with constant $\kappa>0$.
\end{definition}
The application in mind of the previous definition is of 
\Cref{example:robust:precision} where $\bfB^*$ is $\mathcal{J}$-sparse 
and the collection  $O\subset[n]$ of nonzero rows of ${\bfTheta}^*$ 
is sparse. Note that \Cref{def:RE} differs from \Cref{def:RE:vectors} 
with respect to norm of corruption vector ${\bfTheta}$ in the sense 
that a mixed $\ell_2$/$\ell_1$-norm is used. We refer to \citep{balmand:dalalyan2015} for the 
motivation on this regard. 

A major difference between the standard RE condition \eqref{equation:RE:standard} and the augmented version in \Cref{def:RE:vectors} is that the associated dimension reduction cone allows a degree of freedom between the inclinations for 
$\bb$ and $\btheta$. This is encoded in the additional parameter $\gamma$. In particular, the augmented dimension reduction cone in $\re^{p+n}$ is possibly much larger than the Cartesian product of the dimension reduction cones 
\begin{align}
{\calC}_{{S}}(c_1)&:=\left\{\bb\in\re^{p}:
\Vert\bb_{{S}^c}\Vert_{1}\le c_1\Vert\bb_{{S}}\Vert_{1}
\right\}\\
{\calC}_{O}(c_2)&:=\left\{\btheta\in\re^{n}:
\Vert\btheta_{O^c}\Vert_{1}\le c_2\Vert\btheta_{O}\Vert_{1}\right\},
\end{align}
for given $c_1,c_2>1$. Hence, the augmented RE condition is stronger 
than the ``naive'' RE-type condition we can obtain by replacing in 
\Cref{def:RE:vectors} the set ${\calC}_{{S},O}(c,\gamma)$ 
by ${\calC}_{{S}}(c_1)\times {\calC}_{O}(c_2)$.
These are important observations since, in some situations studied in prior work, 
for estimators $(\widehat\bb,\widehat \btheta)$ of unknown parameters $(\bb^*,\btheta^*)
$ it is only proved that the augmented error $(\widehat \bb-\bb^*,\widehat\btheta-
\btheta^*)$ belongs to the large cone of \Cref{def:RE:vectors} with different penalization factors for the coordinates $\bb$ and $\btheta$. The same comments apply to \Cref{def:RE}.  We refer to \citep{nguyen:tran2013,balmand:dalalyan2015} and to 
\Cref{rem:robust:lasso} and \Cref{rem:robust:precision:matrix} for details.

The rest of the paper is organized as follows. We start by two examples 
described in \Cref{sec:examples} that explain our interest for the augmented
RE condition. In \Cref{s:results}, we present our main results for random 
Gaussian matrices along with their consequences on the validity of the augmented 
RE condition. The comparisons with previous work on robust estimation are discussed 
in \Cref{sec:cons}. Theorems stated in \Cref{s:results} are user-friendly versions
of more general results, which are stated in \Cref{sec:full_thms}. 
We provide in \Cref{sec:high_level_proof} a high-level overview of the proofs.
The detailed proofs are postponed to the supplementary material. 

\section{Motivating examples}\label{sec:examples}

To motivate our framework, we give some background on two models studied in the literature on robust estimation \citep{nguyen:tran2013,balmand:dalalyan2015} by convex programming.

\begin{example}[Robust Lasso]\label{example:robust:lasso}
Let $\bfX$ be a $n\times p$ matrix and $\bfX^{(n)}$ be its normalized version. 
In \citep{nguyen:tran2013}, the authors analyze the following linear regression
$$
\by = \bfX^{(n)} \bb^*-\btheta^*+\bw,
$$
where the rows of $\bfX$ are i.i.d.\ ${\calN}_p(\mathbf 0,{\bfSigma})$ 
random vectors, $\btheta^*\in\re^n$ is the contamination and $\bw$ is a 
centered Gaussian noise. The contamination is such that the set 
$O:=\{i:\theta_i\neq 0\}\subset[n]=\{1,\ldots,n\}$ of nonzero coordinates 
is unknown but small. The dimension $p$ is assumed to be large, possibly larger
than $n$. It is also assumed that 
${\bfSigma}\succ0$, i.e., that ${\bfSigma}$ is non-singular.  

As in the standard Lasso \citep{tibshirani1996} corresponding to $\btheta^*=\mathbf 0$, 
the ultimate goal is to estimate the parameter $\bb^*\in\re^p$ 
based solely on the data $(\by,\bfX)$, when $\bb^*$ is supposed to be sparse. 
The proposal of \citep{nguyen:tran2013} is to jointly estimate $(\bb^*,\btheta^*)$ via a robust 
version of the Lasso. The theory of robust sparse recovery for such 
a method relies on the augmented RE condition for $\bfX^{(n)}$.

In \citep{nguyen:tran2013}, an augmented RE-type condition for $\bfX^{(n)}$ is shown 
to hold with high probability under stringent conditions (see \eqref{equation:s:O:constraint} in \Cref{rem:robust:lasso}). One of the 
contributions of this work is to ensure the augmented RE property holds 
under weaker conditions, for a wide class of Gaussian designs. Moreover, 
we obtain much better constants. As commented later, the augmented 
RE constant we get is significantly larger than the one in \citep{nguyen:tran2013} 
(see \Cref{rem:robust:lasso}). As opposed to \citep{nguyen:tran2013}, we include the 
setting where not only the response $\by$ is corrupted but also 
the original design matrix $\bfX$ is subject to contamination. 
Interestingly, in the case of centered Gaussian contamination of the 
design, the obtained augmented RE property is ``essentially'' 
unchanged up to constants independent of $p$.   
\end{example}

\begin{example}[Robust precision matrix estimation]\label{example:robust:precision}
In \citep{balmand:dalalyan2015}, the authors consider the following inference model, which consists 
of $n$ vectors of dimension $p$ gathered in a $n\times p$ matrix $\bfX$ that 
can be represented as 
\begin{equation}
\bfX=\bfY+\bfE,
\label{equation:model}
\end{equation}
where $\bfY\in\re^{n\times p}$ is a random matrix with i.i.d. ${\calN}_p(
\mathbf 0,{\bfSigma})$ rows and $\bfE\in\re^{n\times p}$ is the corruption matrix. Here, the population covariance matrix ${\bfSigma}\in\re^{p\times p}$ and the corruption matrix are supposed to be unknown. The nonzero rows of $\bfE$, corresponding
to corrupted rows of $\bfX$, belong to a small set $O\subset[n]$. 

Assuming ${\bfSigma}$ is nonsingular, \citep{balmand:dalalyan2015} proposes a method for estimating the precision matrix $\mathbf{\Omega}^*:={\bfSigma}^{-1}$, which is relevant for inference in graphical models \citep{meinshausen:buhlmann2006}. The 
estimator proposed in \citep{balmand:dalalyan2015} is shown to 
have a small error provided that $\bfX^{(n)}$ satisfies the augmented RE condition.
We show below that such a property holds with high probability assuming only that 
${\bfSigma}$ is non-singular (see \Cref{rem:robust:precision:matrix} below). 
\end{example}

\section{Main results}\label{s:results}

The following general theorems are the key results for establishing 
the augmented RE condition for the design $\bfX^{(n)}$. For increased 
generality and a larger scope of applicability, we allow the contamination 
to have a random component. More precisely, our contamination model is as follows. 

\begin{assump}[Random contamination model]\label{assump:random:contamination:model}
Suppose \eqref{equation:model} holds, where $\bfY\in\re^{n\times p}$ 
is a random matrix with i.i.d.\ ${\calN}_p(\mathbf 0,{\bfSigma})$ rows. 
Assume $\bfE:=\bfE_{\mathcal{D}}+\bfE_{\mathcal{R}}$, where 
$\bfE_{\mathcal{D}}\in\re^{n\times p}$ is an arbitrary matrix\footnote{In the 
applications in robust estimation the matrix $\bfE_{\mathcal{D}}$ is assumed
to be row-sparse. But such  a condition is not needed for proving RE-type 
conditions.} and $\bfE_{\mathcal{R}}\in\re^{n\times p}$ is a random matrix 
with i.i.d.\ ${\calN}_p(\mathbf 0,{\bfSigma}_{E})$ rows, where ${\bfSigma}_{E}\in\re^{p\times p}$. $\bfE_{\mathcal{R}}$ is independent of $\bfY$. If $\bfE_{\mathcal{D}}\neq0$, we suppose ${\bfSigma}$ is non-singular.
\end{assump}
Although stated in slightly more general form for brevity of presentation, we are mainly interested in two set-ups in Assumption \ref{assump:random:contamination:model}. The first is the set-up $\bfE_{\mathcal{D}}\ne0$ and $\bfE_{\mathcal{R}}=0$ which includes a general deterministic or non-Gaussian contamination model. The second is the case where the rows of the design are corrupted by (possibly non-centered) Gaussian 
${\calN}_p(\bmu_E,{\bfSigma}_{E})$ outliers. This is the case where, for every $i\in O$ and for some $\bmu_E\in\re^p$, $(\bfE_{\mathcal{D}})_{i,\bullet}=\bmu_E$. As discussed later, our obtained results are sharper for Gaussian contamination. This is the main reason for considering two cases.  

In the sequel, it will be useful to define some quantities. We define the matrices
$$
\bfX_{\mathcal{R}}:=\bfY+\bfE_{\mathcal{R}},
\quad\quad
{\bfSigma}_S:={\bfSigma}+{\bfSigma}_{\bfE}.
$$
Set 
$$
\varrho({\bfSigma}_S):=\Big\{\max_{j\in[p]}({\bfSigma}_S)_{jj}\Big\}^{1/2}.
$$ 
\begin{thm}[The vector case]\label{theorem:special:vectors}
Suppose \Cref{assump:random:contamination:model} holds. 
For all $n\ge 208$, with probability at least 
$1-2\exp(-n/297)$, for every $[\bb;\btheta]\in\re^{p+n}$ we have
\begin{align}
\big\Vert\bfX^{(n)}\bb-\btheta\big\Vert_2&\ge
\big(0.24-\|\bfE^{(n)}_{\mathcal{D}}{\bfSigma}_S^{\dagger/2}\|\big)
\big\Vert\big[{\bfSigma}^{1/2}_S\bb;\btheta\big]\big\Vert_2\\
&-36\varrho({\bfSigma}_S)\Vert 
\bb\Vert_1\sqrt{\nicefrac{\log p}{n}}\\ 
&- 33\Vert\btheta\Vert_1\sqrt{\nicefrac{\log n}{n}}.
\end{align}
\end{thm}
\Cref{theorem:special:vectors} generalizes Theorem 1 in 
\cite{raskutti:wainwright:yu2010} to the augmented space $\re^{p+n}$ under 
the design contamination model given by 
\eqref{equation:model} and Assumption \ref{assump:random:contamination:model}. 
It generalizes Theorem 1 in \cite{raskutti:wainwright:yu2010} in the sense that it can be used in the analysis of regression models corrupted by outliers either in the response variable 
and/or in the design matrix. In particular, when there either deterministic 
outilers or random outlier are absent, \textit{i.e.}, either $\bfE_{\mathcal D} 
= \mathbf 0$ or $\bfSigma_S = \mathbf 0$, then the inequality of 
\Cref{theorem:special:vectors} takes a much simpler form (very close to
the one of Theorem 1 in \cite{raskutti:wainwright:yu2010}), since $\|\bfE^{(n)}_{\mathcal{D}}
{\bfSigma}_S^{\dagger/2}\|=0$.

Besides guaranteeing an augmented RE condition, \Cref{theorem:special:vectors}  
provides an ``augmented transfer principle'' \citep{oliveira2013,oliveira2016,rudelson:zhou2013} in the setting of corrupted data with Gaussian design, which may be useful in other corrupted models. The core of the argument is to avoid singular values bounds of $\bfX$, or mutual incoherence properties between column-spaces of $\bfX$ and $\bfI_n$, and to extend the techniques of 
\cite{raskutti:wainwright:yu2010} to the mentioned set-up of contaminated data. A key strategy in that quest is to simultaneously obtain lower and upper bounds of two random bilinear forms with very different behaviors. We refer to \Cref{sec:high_level_proof} for a discussion on the challenges found in such argument.

Let us now state the corresponding result for matrices.

\begin{thm}[The matrix case]\label{theorem:special}
Suppose \Cref{assump:random:contamination:model} holds. For all
$n\ge 208$, with probability at least $1-2\exp(-n/297)$, for all 
$[\bfB;{\bfTheta}]\in\re^{(p+n)\times p}$, we have
\begin{align}
\big\Vert\bfX^{(n)}\bfB-\bfTheta\big\Vert_{2,2}
\ge& \big({0.24-\|\bfE_{\mathcal{D}}^{(n)}{\bfSigma}_S^{\dagger/2}\|}\big)
\left(\Vert{\bfSigma}^{1/2}_S\bfB\Vert_{2,2}\vee\Vert{\bfTheta}\Vert_{2,2}\right)\\
&-36\varrho({\bfSigma}_S)\Vert \bfB\Vert_{1,1}
\sqrt{\nicefrac{\log p}{n}}\\
&-33\Vert{\bfTheta}\Vert_{2,1}\sqrt{\nicefrac{\log n}{n}}.
\end{align}
\end{thm}
\Cref{theorem:special} extends Theorem 1 in \cite{raskutti:wainwright:yu2010} in different directions. First, by taking ${\bfTheta}=\mathbf 0$, \Cref{theorem:special} establishes a lower bound over matrix parameters $\bfB\in\re^{p\times p}$, which entails an RE-type condition for multivariate models. To the best of our knowledge, there was no such result in the
literature. Second, \Cref{theorem:special} establishes a lower bound over the augmented matrix parameters $[\bfB;{\bfTheta}]\in\re^{(p+n)\times p}$ under the design contamination model given by \eqref{equation:model} and  \Cref{assump:random:contamination:model}. Note that this is new and appealing even in the 
case $\bfE = \mathbf 0$, which is relevant for checking RE-type properties for regression models where only the multivariate response is contaminated (but not the design) by outliers. Finally, it its full generality, \Cref{theorem:special} implies 
an augmented RE condition for the multidimensional linear models in which both the response and the design are contaminated. 

We next present important consequences that the augmented RE-type conditions of \Cref{def:RE:vectors} and \Cref{def:RE} holds for $\bfM=\bfX^{(n)}$ with high probability 
and with reasonable sample size as long ${\bfSigma}$ is sufficiently 
well-posed. For simplicity, we state the corollary only for the matrix case, the 
vector case being similar and even easier to handle.

\begin{cor}[Well-posedness of ${\bfSigma}$ implies the augmented RE-condition]
\label{cor:RE-condition}
Grant \Cref{assump:random:contamination:model} with a non-degenerate matrix 
$\bfSigma$. Let $n\ge208$  and 
\begin{align}
\gamma \ge 1.1\sqrt{\nicefrac{\log p}{\log n}}.
\label{equation:cor:penalization:factors}
\end{align}
Set $\kappa:=1\wedge\sqrt{\sigma_{\min}(\bfSigma+\bfSigma_{\bfE})}$. For some $c_0>0$, suppose the sample size satisfies\footnote{In case 
$
\gamma \ge 1.1\varrho(\bfSigma_S)\sqrt{\nicefrac{\log p}{\log n}},
$
condition \eqref{equation:cor:n} can be replaced by
$
36\left(\gamma\sqrt{s}+\sqrt{o}\right)\sqrt{\frac{\log n}{n}}\le \kappa c_0.
$
}
\begin{align}
&36\varrho({\bfSigma}_S)
\left(\gamma\sqrt{s}+\sqrt{o}\right)\sqrt{\frac{\log n}{n}}\le \kappa c_0,
\label{equation:cor:n}\\
&\mathsf{c}_n:=0.24-\|\bfE_{\mathcal{D}}^{(n)}{\bfSigma}_S^{\dagger/2}\|
-(1+c)c_0>0.\label{equation:cor:c}
\end{align}
Then, with probability at least 
$1-2\exp(-n/297)$, $\bfX^{(n)}$ satisfies $\RE_{s,o}^{\rm mat}(c,\gamma)$ 
with the constant $\mathsf{c}_n\kappa$.
\end{cor}
\begin{proof}
Let $\Omega_0$ be the event of probability at least $1-2e^{-n/297}$ in which
the claims of \Cref{theorem:special} hold true. The rest of the proof
contains only deterministic bounds assuming that $\Omega_0$ is realized. 
Let $\mathcal{J}$ and $O$ be as in \Cref{def:RE} with
$|\mathcal{J}|=s$ and $|O|=o$.  Let $[\bfB;{\bfTheta}]\in\re^{(p+n)\times p}$. 
On the one hand, the Cauchy-Schwarz inequality yields 
$\Vert\bfB_{\mathcal{J}}\Vert_{1,1}\le\sqrt{|\mathcal{J}|}\Vert\bfB\Vert_{2,2}$ and  
$\Vert{\bfTheta}_{O,\bullet}\Vert_{2,1}\le\sqrt{|O|}\Vert{\bfTheta}\Vert_{2,2}$. 
On the other hand, we obviously have
\begin{align}
\Vert\bfB\Vert_{1,1}&=\Vert\bfB_{\mathcal{J}^c}\Vert_{1,1}
+\Vert\bfB_{\mathcal{J}}\Vert_{1,1},\\
\Vert{\bfTheta}\Vert_{2,1}&=\Vert{\bfTheta}_{O^c,\bullet}\Vert_{2,1}
+\Vert{\bfTheta}_{O,\bullet}\Vert_{2,1}.
\end{align}
Assuming that $[\bfB;{\bfTheta}]\in{\calC}_{\mathcal{J},O}
(c,\gamma)$, the previous relations lead to
\begin{align}
\gamma\Vert\bfB\Vert_{1,1}&+\Vert{\bfTheta}\Vert_{2,1}\\
&\le (c+1)\big(\gamma\sqrt{|\mathcal{J}|}\,\Vert\bfB\Vert_{2,2}+
\sqrt{|O|}\,\Vert{\bfTheta}\Vert_{2,2}\big)\\
&\le (c+1)\big(\gamma\sqrt{s}+\sqrt{o}\big)
\big(\Vert\bfB\Vert_{2,2}\vee\Vert{\bfTheta}\Vert_{2,2}\big).
\end{align}
If we set $\lambda:=\sqrt{\nicefrac{\log n}{n}}$ for simplicity of notation, we have 
\begin{align}
&36\varrho({\bfSigma}_S)\Vert \bfB\Vert_{1,1}
\sqrt{\nicefrac{\log p}{n}}
+33\Vert{\bfTheta}\Vert_{2,1}\sqrt{\nicefrac{\log n}{n}}\\
&\le 36\varrho({\bfSigma}_S)\Vert \bfB\Vert_{1,1}
\sqrt{\nicefrac{\log p}{n}}
+36\Vert{\bfTheta}\Vert_{2,1}\lambda\\ 
&\le 36(\gamma\varrho({\bfSigma}_S)\Vert \bfB\Vert_{1,1}
+\Vert{\bfTheta}\Vert_{2,1})\lambda\\
&\le 36\lambda(c+1)\varrho({\bfSigma}_S)
\left(\gamma\sqrt{s}+\sqrt{o}\right)\left(\Vert\bfB\Vert_{2,2}\vee\Vert{\bfTheta}\Vert_{2,2}\right) \\
&\le (1+c)c_0\kappa\left(\Vert\bfB\Vert_{2,2}\vee\Vert{\bfTheta}\Vert_{2,2}\right),
\label{cor:eq1}
\end{align}
where in the last inequality we have used condition \eqref{equation:cor:n}.

Since $\bfSigma$ is nondegenerate, the same holds for $\bfSigma_S\succeq \bfSigma$. 
Hence, 
\begin{align}
\Vert\bfSigma_S^{1/2}\bfB\Vert_{2,2}\vee\Vert{\bfTheta}\Vert_{2,2}
&\ge (\sigma_{\min}(\bfSigma_S)^{1/2}\Vert\bfB\Vert_{2,2})\vee\Vert{\bfTheta}\Vert_{2,2}\\
&\ge \kappa \Vert\bfB\Vert_{2,2}\vee\Vert{\bfTheta}\Vert_{2,2}.\label{cor:eq2}
\end{align}
Combining the claim of \Cref{theorem:special} with \eqref{cor:eq1},\eqref{cor:eq2} and 
the definition of $\mathsf{c}_n$, see \eqref{equation:cor:c}, we get
\begin{align}
\|\bfX^{(n)}\bfB-\bfTheta\|_{2,2} \ge \mathsf{c}_n\kappa \Vert\bfB\Vert_{2,2}\vee\Vert{\bfTheta}\Vert_{2,2}.
\label{cor:eq3}
\end{align}
This completes the proof of the corollary.
\end{proof}
Note that \Cref{theorem:special} and \Cref{cor:RE-condition} lead to meaningful 
results only when the spectral norm of the matrix $\bfE_{\mathcal D}^{(n)}\bfSigma_S^{\dagger/2}$ is small (for instance, smaller than $0.2$). As noted in \citep{balmand:dalalyan2015}, 
this kind of condition is not too restrictive. Indeed, one can perform a first round
of data preprocessing consisting in removing the most striking outliers.  After such a
preprocessing, it is reasonable to assume that $\|\bfE_{\mathcal D}^{(n)}
\bfSigma_S^{\dagger/2}\|\le 0.2$. 

To complete this section, let us just mention that the condition of nondegeneracy of
$\bfSigma$ can be further relaxed. Indeed, it follows from the proof of 
\Cref{cor:RE-condition} that it suffices to assume that the matrix $\bfSigma_S$ fulfills 
the following RE-type condition: For every $[\bfB;\bfTheta]$ from $\calC_{\mathcal J,O}$ 
with $|\mathcal J| =s$ and $|O| = o$, it holds that
\begin{align}
\|\bfSigma_S^{1/2}\bfB\|_{2,2} + \|\bfTheta\|_{2,2} \ge \kappa 
(\|\bfB\|_{2,2}\vee \|\bfTheta\|_{2,2}).
\end{align}

\section{Consequences in robust estimation}\label{sec:cons}

The goal of this section is to compare \Cref{cor:RE-condition} to the 
similar results previously obtained in the examples described in 
\Cref{sec:examples}.

\subsection{Augmented RE for the Robust Lasso}\label{rem:robust:lasso}

Let us start by considering the Robust Lasso estimator proposed in 
\citep{nguyen:tran2013} to perform sparse recovery in presence of outliers, see 
\Cref{example:robust:lasso}. Throughout this discussion we assume that 
$\varrho({\bfSigma})=1$, which significantly simplifies the comparison. 
It is established in \citep{nguyen:tran2013} that the augmented RE condition 
of \Cref{def:RE:vectors} with the parameter
\begin{align}
\gamma^2 = \Theta\Big(\frac{\log p}{\log n}\Big),
\label{equation:lambda:gamma:robust:lasso}
\end{align}
holds with high probability for Gaussian matrices $\bfX^{(n)}$, provided that 
${\bfSigma}$ is nondegenerate, \textit{i.e.}, $\sigma_{\min}:
=\sigma_{\min}(\bfSigma)>0$, and that the parameters $(s,o,n)$ are constrained 
by\footnote{Although constraint \eqref{equation:s:O:constraint} does not
appear in the statement of Lemma 1 in \citep{nguyen:tran2013}, it is indeed required
in the proof of Lemma 1. In fact, Lemma 1 relies on Lemma 2, which requires 
\eqref{equation:s:O:constraint}.}
\begin{align}
&o=O\Big(\frac{s\log p}{\sigma_{\min}\log n}\Big),
\label{equation:s:O:constraint}\\
&\sigma_{\max}(\bfSigma) = \Theta(1),\label{equation:max:constraint}\\
&{\frac{s\log p}{\sigma_{\min}^2 n}}\bigvee{\frac{o\log n}{\sigma_{\min}n}}\le c_1,
\label{equation:n:robust:lasso}
\end{align}
for sufficiently small $c_1$. The obtained RE constant is 
$\kappa'={\Omega}(1)\min\{\sqrt{\sigma_{\min}({\bfSigma})},1\}$. See Lemmas 1-2 in \citep{nguyen:tran2013}. Additionally, if $\bb^*$ is $s$-sparse and $\btheta^*$ is 
$o$-sparse and the 
above RE condition is fulfilled, risk bounds---minimax rate-optimal when 
$o=0$---in the $\ell_2$-norm are guaranteed. 

If we compare these results of \citep{nguyen:tran2013} to that of 
\Cref{cor:RE-condition}, we can make four observations. We first remark 
that our result does not require
any assumption of type \eqref{equation:s:O:constraint} for the sparsity 
levels $(s,o)$ of $(\bb^*,\btheta^*)$. Indeed, for our result to hold 
there is no need to put any constraint relating the sparsity level of 
the unknown parameter to the number of outliers. 

The second observation is that our condition \eqref{equation:cor:n} 
is of the same flavor as \eqref{equation:n:robust:lasso}, but it holds without
the additional assumption \eqref{equation:max:constraint}. The latter is not
suitable in a high-dimensional setting, since in many situations the largest
eigenvalue of a $p\times p$ matrix increases with the dimension $p$.

The third observation is that our result provides improved numerical constants. 
For instance, when $\bfSigma$ is the identity matrix, the results in 
\citep{nguyen:tran2013} yield\footnote{We use the informal notation $\nearrow$ 
to express the fact that some quantity can be in some regime very close to some value.
} $\kappa \nearrow 0.0625$ with 
probability $\ge1-c_1\exp(-c_2n)$ for $n\ge6400$. Our result leads to 
$\kappa\nearrow 0.24$ with probability $\ge1-2\exp(-n/297)$ for $n\ge208$.

Finally, our results are valid for corrupted designs ($\bfE\neq0$), a case that does 
not enter into the scope of \citep{nguyen:tran2013}. In addition, we can handle not only
deterministic but also random outliers, with improved bounds for Gaussian contamination of the design matrix.\footnote{Our bounds do not depend on $\sigma_{\max}(\mathbf{\Sigma}_{\mathbf{E}})$ which is potentially $\mathcal{O}(p)$ when the coordinates of outliers are highly correlated.}

\subsection{Augmented RE for \Cref{example:robust:precision}}
\label{rem:robust:precision:matrix}

Next, we present an application of \Cref{cor:RE-condition} to the problem
of matrix estimation addressed in \citep{balmand:dalalyan2015} and 
briefly discussed in \Cref{example:robust:precision}. More precisely,
\citep{balmand:dalalyan2015} proposes a robust estimator of the precision matrix 
$\mathbf{\Omega}^*:={\bfSigma}^{-1}$ that can be computed by convex programming. 
The estimator is analyzed in the high-dimensional setting, where 
$\mathbf{\Omega}^*\in\re^{p\times p}$ is $\mathcal{J}$-sparse with 
the support $\mathcal{J}:=\{J_j:j\in[p],J_j\subset[p]\}$ and 
$p\le|\mathcal{J}|\ll n\wedge p^2$. Theorem 3 in \citep{balmand:dalalyan2015}
provides a risk bound for the aforementioned estimator, provided that the 
difference between the estimated and the true matrix satisfies the augmented
RE condition with $\gamma = 1$ and $c=2$. In \citep{balmand:dalalyan2015}, 
however, the augmented RE property is assumed to hold a priori.

One contribution of \Cref{cor:RE-condition} is to complement the findings of 
\citep{balmand:dalalyan2015} by showing that, under the same conditions on the 
sample size as in \citep{balmand:dalalyan2015}, the augmented 
RE condition is indeed satisfied with high-probability.

\section{General versions of the main theorems}\label{sec:full_thms}

\Cref{theorem:special:vectors} and \Cref{theorem:special} are user friendly versions
of more general theorems that we will state in the present section. The main advantage
of these general versions is that they can lead to improved numerical constants in
some concrete situations. For instance, if we know that the sample size is larger than 
$10^4$, we can get some improvement in the RE constant or in the probability of 
validity of the RE condition.  

To state the results of this section, we need some additional notations. 
In fact, all the numerical constants appearing in \Cref{theorem:special:vectors} and \Cref{theorem:special} are obtained by instantiating 5 parameters $\epsilon$, $\alpha$, $\beta$, $\sigma$ and $\tau$. These parameters should be positive with the only constraint
that $\epsilon<3/4$. With the help of these parameters, we introduce the following
quantities:
\begin{align}
\rho:=(1+\tau)(1+1/\sigma),
\quad
\mu_\epsilon:=1-\frac{(3/4-\epsilon)}{\sqrt{2}}
\label{equation:def:rho:mu:epsilon}
\end{align}
and, for any matrix $\bfA \in\re^{n\times p}$, 
\begin{align}
\mathsf{C}_{n}(\bfA )&:=\rho\left[\frac{(3/4-\epsilon)}{\sqrt{2}}
(1-e^{-\frac{n\epsilon^2}{2}})-(1-\rho^{-1})\right]\\
&-\left(\frac{1}{2\alpha}+\frac{1}{2\beta}+
\frac{\sigma_{\max}(\bfA )}{\sqrt{n}}\right).
\label{equation:def:CA}
\end{align}

\begin{thm}[The vector case]\label{theorem:general:vectors}
Grant \Cref{assump:random:contamination:model}. Let $\epsilon\in(0,3/4)$ and 
$\alpha$, $\beta$, $\sigma$ and $\tau$ be positive numbers. Then, for all
$
n\ge\left(\frac{2\sigma^2\log 2}{(1+\tau)^2\mu_\epsilon^2}\right)\vee10,
$
with probability at least 
$
1-2\exp\left[-\frac{(1+\tau)^2\mu_{\epsilon}^2}{2\sigma^2}n\right],
$
we have
\begin{align}
\big\Vert\bfX^{(n)}\bb-\btheta\big\Vert_2&\ge
\mathsf{C}_{n}(\bfE_{\mathcal{D}}{\bfSigma}_S^{\dagger/2})
\left\Vert\left[{\bfSigma}^{1/2}_S\bb;\btheta\right]\right\Vert_2\\
&-\rho\left(2+\alpha\sqrt{2}\right)\varrho({\bfSigma}_S)\Vert \bb\Vert_1\sqrt{\frac{\log p}{n}}\\ 
&- \rho\beta\Vert\btheta\Vert_1\sqrt{\frac{2\log n}{n}}.
\end{align}
\end{thm}

We note that, \Cref{theorem:special:vectors} is obtained by choosing 
in \Cref{theorem:general:vectors} $\epsilon=0.19$, $\tau=0.02$, 
$\sigma=7.5$ and $\alpha=\beta=20$. Of course, one can choose other 
values for these parameters and obtain slightly different versions
of \Cref{theorem:special:vectors}. In particular, if we know in advance
that the sample size is large, we can choose a much smaller $\epsilon$
and larger $\sigma$. This will lead to a value of $\mathsf{C}_n$ that
is closer to $3/(4\sqrt{2})\approx 0.53$.

Let us state now the general version of the main theorem in the matrix case.

\begin{thm}[The matrix case]\label{theorem:general}
Suppose Assumption \ref{assump:random:contamination:model} holds. Set $\epsilon\in(0,3/4)$ and positive numbers $\alpha$, $\beta$, $\sigma$ and $\tau$. Then, for all
$
n\ge\left(\frac{2\sigma^2\log 2}{(1+\tau)^2\mu_\epsilon^2}\right)\vee10,
$
with probability at least 
$
1-2\exp\left[-\frac{(1+\tau)^2\mu_{\epsilon}^2}{2\sigma^2}n\right],
$
we have
\begin{align}
\left\Vert \bfX^{(n)}\bfB-\bfTheta\right\Vert_{2,2}&\ge 
{\mathsf{C}_{n}(\bfE_{\mathcal{D}}{\bfSigma}_S^{\dagger/2})}
\left(\Vert{\bfSigma}^{1/2}_S\bfB\Vert_{2,2}\vee\Vert{\bfTheta}\Vert_{2,2}\right)\\
&-\rho\left(2+\alpha\sqrt{2}\right)\varrho({\bfSigma}_S)\Vert \bfB\Vert_{1,1}\sqrt{\frac{\log p}{n}}\\
&-\rho\beta\Vert{\bfTheta}\Vert_{2,1}\sqrt{\frac{2\log n}{n}}.
\end{align}
\end{thm}

The constant $3/4$ in \Cref{theorem:general:vectors} and \Cref{theorem:general} is taken for simplicity. Its derivation is based on the fact that for a standard Gaussian vector 
$\mathbf{g}\in\re^{n}$, $\esp[\Vert\mathbf{g}\Vert_2]=\sqrt{n}+o(\sqrt{n})$. It can be replaced by any constant strictly less than 1 for a sufficiently large $n$.

\section{Outline of the proofs}\label{sec:high_level_proof}

This section provides a high level overview of the proof of \Cref{theorem:general:vectors} and \Cref{theorem:general}. The detailed proofs can be found in the supplementary
material. 

\Cref{theorem:general:vectors} requires establishing a lower bound, for any given $[\bb;\btheta]\in\re^{p+n}$, on
$
\left\Vert(\mathbf{Y}^{(n)}+\bfE^{(n)})\bb-\btheta\right\Vert_2^2.
$
Clearly, it is not enough to establish a lower bound for $\Vert\mathbf{Y}^{(n)}\bb\Vert^2$ as one also needs to simultaneously control terms of the form
\begin{eqnarray}
\Vert\mathbf{E}^{(n)}\bb\Vert^2+\Vert\btheta\Vert_2^2, &\quad & (\mathbf{Y}^{(n)}\bb)^\top\bfE^{(n)}\bb,\label{equation:cross:terms:fast}\\
\btheta^\top\mathbf{Y}^{(n)}\bb, &&\btheta^\top\mathbf{E}^{(n)}\bb.
\label{equation:cross:terms:slow}
\end{eqnarray}
In particular, the bound we seek will depend on the interactions between the column-spaces of the random matrices $\mathbf{Y}^{(n)}$ and $\bfE^{(n)}$, of $\mathbf{I}_n$ and $\mathbf{Y}^{(n)}$ and of $\mathbf{I}_n$ and $\mathbf{E}^{(n)}$. In this regard, we note that the decomposition $\bfE^{(n)}=\bfE_{\mathcal{D}}^{(n)}+\bfE_{\mathcal{R}}^{(n)}$ adds additional complexity (see Assumption \ref{assump:random:contamination:model}).

In the set-up where $\bfE\equiv0$, the authors of \cite{nguyen:tran2013} make use of singular value bounds and mutual incoherence arguments between column-spaces of $\mathbf{Y}^{(n)}$ and  $\mathbf{I}_n$. In order to obtain the sharper results mentioned in Remark \ref{rem:robust:lasso} and handle the presence of $\mathbf{E}$ in an optimized way, we avoid such type of approach. Instead, we adopt an approach inspired by 
\cite{raskutti:wainwright:yu2010}. We recall that the bounds of 
\cite{raskutti:wainwright:yu2010} are valid only for the set-up with 
no contamination ($\bfE\equiv0,\btheta\equiv0$). In their case, terms \eqref{equation:cross:terms:fast}-\eqref{equation:cross:terms:slow} are zero so that it is sufficient to establish a single uniform lower bound in high probability for the random quadratic form $\bb\mapsto\Vert\mathbf{Y}^{(n)}\bb\Vert^2$. 

The first step of the proof is to obtain a bound in expectation (see Subsection \ref{ss:expectation}).  
\subsection{Bound in expectation}
Instead of considering the whole augmented space $\re^{p+n}$, the first major step in the proof is to obtain, for every $r_1,r_2>0$, a lower bound for
\begin{eqnarray}
\esp\left[\inf_{[\bb;\btheta]\in V(r_1,r_2)}\Vert\mathbf{X}^{(n)}\bb-\btheta\Vert_2\right],
\label{equation:inf:V:r1:r2}
\end{eqnarray}
where
\begin{eqnarray*}
V(r_1,r_2):=\left\{[\bb;\btheta]\in\re^{p+n}:
\begin{array}{l}
\Vert\mathbf{\Sigma}_S^{1/2}\bb\Vert_2^2+\Vert\btheta\Vert_2^2=1,\\
\Vert \bb\Vert_1\le r_1,\Vert\btheta\Vert_1\le r_2
\end{array}
\right\}.
\end{eqnarray*}
The control of \eqref{equation:inf:V:r1:r2} is done via empirical process techniques. 
Let us first make some comments regarding the definition of $V(r_1,r_2)$ and then point challenges found in our approach.
\begin{itemize}\parsep=0pt
\item The normalization 
\begin{eqnarray}
\left\Vert\left[\begin{array}{cc}
\mathbf{\Sigma}_S^{1/2} & \mathbf{0}_{p\times n}\\
\mathbf{0}_{n\times p} & \mathbf{I}_n
\end{array}\right]
\left[\begin{array}{c}
\bb\\
\btheta
\end{array}
\right]
\right\Vert_2=1
\label{equation:normalization:aug}
\end{eqnarray}
used in $V(r_1,r_2)$ defines an ellipse in the \emph{augmented space} $\re^{p+n}$. As it will be clear in the proof, this is essential in order to extend the obtained lower bound over $V(r_1,r_2)$ to the whole augmented space using the homogeneity property of the norm. For instance, it seems unlikely that our technique would work if we considered naively the Cartesian product of the ellipses $\{\bb\in\re^p:\Vert\mathbf{\Sigma}_S^{1/2}\bb\Vert_2^2=1/{2}\}$ and $\{\btheta\in\re^n:\Vert\btheta\Vert_2^2=1/{2}\}$. On the other hand, it is harder to establish uniform lower bounds with the constraint $\Vert\mathbf{\Sigma}_S^{1/2}\bb\Vert_2^2+\Vert\btheta\Vert_2^2=1$ since more degrees of freedom are allowed: one of
the two norms $\Vert\mathbf{\Sigma}_S^{1/2}\bb\Vert_2$ and $\Vert\btheta\Vert_2$ can 
be zero.

\item Note that we use $\bfSigma_S=\bfSigma+\bfSigma_{\mathbf{E}}$, rather than $\bf\Sigma$, in the definition of $V(r_1,r_2)$. The intuition is that $\bfSigma_S$ is the covariance matrix of the rows of $\mathbf{X}_{\mathcal{R}}=\mathbf{Y}+\bfE_{\mathcal{R}}$. Indeed, $\mathbf{Y}$ and $\bfE_{\mathcal{R}}$ are independent. Hence, $V(r_1,r_2)$ ``incorporates'' the interaction between the design $\bfY$ and the random corruption $\bfE_{\mathcal{R}}$. This will provide a tighter control of $\bfE_{\mathcal{R}}$: our bounds will not depend on $\sigma_{\max}({\bfSigma}_{\bfE})$, which is potentially  $\mathcal{O}(p)$ when the
coordinates of outliers are highly correlated.
\end{itemize}

As already mentioned, a challenge in our approach is to establish a deterministic 
lower bound for the random form $\bb\mapsto\Vert\mathbf{Y}^{(n)}\bb\Vert_2^2$ and, at the same time, to take into account all the random terms in \eqref{equation:cross:terms:fast}-\eqref{equation:cross:terms:slow}. At the heart of our analysis is a ``\emph{splitting argument}'', which we explain next (see Lemmas \ref{lemma:expectation}-\ref{lemma:aux2} in the supplementary material). We split the task of lower bounding 
\eqref{equation:inf:V:r1:r2} into three parts, thanks to the inequality  
\begin{align}
\inf_{[\bb;\btheta]\in V(r_1,r_2)} &\Vert\mathbf{X}^{(n)}\bb-\btheta\Vert_2\\
&\ge 
\frac{\mathsf{I}_1\wedge 1}{\sqrt{2}}
-\sqrt{2\mathsf{I}_2}
-\frac{\sigma_{\max}(\bfE_{\mathcal{D}}{\bfSigma}_S^{\dagger/2})}{\sqrt{n}},
\label{equation:splitting}
\end{align}
where 
$
\mathsf{I}_1:=\inf_{\bb\in V_1\left(\sqrt{2}r_1\right)}\Vert\mathbf{X}_{\mathcal{R}}^{(n)}\bb\Vert_2
$
and
\begin{align}
\mathsf{I}_2:=\sup_{[\bb;\btheta]\in V(r_1,r_2)}\btheta^\top\mathbf{X}_{\mathcal{R}}^{(n)}\bb.
\end{align}
In the above formula, 
\begin{align}
V_1(r_1):=\left\{\bb\in\re^p:\Vert\mathbf{\Sigma}_S^{1/2}\bb\Vert_2=1,\Vert \bb\Vert_1\le r_1\right\}.
\end{align}
The splitting \eqref{equation:splitting} is justified as follows.  The third term 
in \eqref{equation:splitting} allows us to separate the impact of the (possibly
deterministic) corruption $\bfE_{\mathcal{D}}$. 
The first and the second terms in \eqref{equation:splitting} account for randomness in a precise way. Note that  the random quadratic form 
\begin{align}
V_1(\sqrt{2}r_1)\ni\bb\mapsto\Vert\mathbf{X}_{\mathcal{R}}^{(n)}\bb\Vert_2^2
\label{equation:random:quad:form}
\end{align}
in $\mathsf{I}_1^2$ is much simpler to handle: it does not depend on $\btheta$ and it is defined by a random matrix whose i.i.d. rows have distribution $\mathcal{N}_p(0,\mathbf{\Sigma}_S)$. Finally, a crucial point is that an \emph{equality} constraint $\Vert\mathbf{\Sigma}_S^{1/2}\bb\Vert_2=1$ is maintained in 
$\mathsf{I}_1$. Although this causes a deterioration of numerical constants  
(approximately by factor $\sqrt{2}$), this appears to be crucial for getting a uniform lower bound of \eqref{equation:random:quad:form} {with high probability}\footnote{
This is done by means of Gordon's comparison inequality and Gaussian concentration inequality.} 
and to obtain a positive RE constant. Moreover, this technique will give a simpler way to handle the terms of the form \eqref{equation:cross:terms:fast} in \eqref{equation:splitting} which account for the interactions between the column-spaces of $\mathbf{E}_{\mathcal{R}}$ and $\mathbf{Y}$.

The second term $\sqrt{2\mathsf{I}_2}$ in \eqref{equation:splitting} only depends on terms of the form \eqref{equation:cross:terms:slow} which account for the interactions between the column-spaces of $\mathbf{Y}$ and $\mathbf{I}_n$ and of $\mathbf{E}_{\mathcal{R}}$ and $\mathbf{I}_n$. Moreover, differently than $\mathsf{I}_1$, the random bilinear form 
\begin{eqnarray}
V(r_1,r_2)\ni[\bb,\btheta]\mapsto\btheta^\top\mathbf{X}^{(n)}_{\mathcal{R}}\bb
\label{equation:random:bilinear:form}
\end{eqnarray}
found in $\mathsf{I}_2$ can be controlled by means of the \emph{relaxed} constraints $\Vert\mathbf{\Sigma}_S^{1/2}\bb\Vert_2\le1$ and $\Vert\btheta\Vert_2\le1$. Nevertheless, a major drawback of $\sqrt{2\mathsf{I}_2}$ when compared to $\mathsf{I}_1$ is the fact that it has a $\mathcal{O}(n^{-1/2})$-worst decay in the sample size $n$. In that respect, a tight simultaneous control between $\mathsf{I}_1$ and $\sqrt{\mathsf{I}_2}$ in \eqref{equation:splitting} must be properly balanced so that tight rates in $(s,o,n,p)$ are obtained.\footnote{Note that this has implications on the minimax rate of convergence of associated estimators (see Example \ref{example:robust:lasso} and Remark \ref{rem:robust:lasso}).} We use a bias reduction argument in order to achieve this goal: we establish a uniform upper bound \emph{in expectation} of \eqref{equation:random:bilinear:form} via a different use of Slepian's inequality over a ``correct'' set. As a result, $\bb$ and $\btheta$ can be decoupled.\footnote{A naive control by means of Cauchy-Schwarz would give a decay of one order worst.}

\subsection{Further details}
Once we obtain a bound in expectation for \eqref{equation:inf:V:r1:r2}, a next step is to establish a variance control and concentration (see Subsection \ref{ss:concentration} in the supplemental material). We give some remarks in that respect. The corruption parameter $\btheta$ acts as a ``bias'' and, hence, it has no impact on the variance. The same observation does not hold for the corruption matrix $\mathbf{E}_{\mathcal{R}}$. We will crucially use independence between $\mathbf{Y}$ and $\mathbf{E}_{\mathcal{R}}$.

A final step of the proof of Theorem \ref{theorem:general:vectors} is to extend the obtained concentration inequality restricted over $V(r_1,r_2)$ to the whole augmented space $\re^{p+n}$. This will be done using a standard peeling argument in subsection \ref{ss:peeling} of the supplemental material. A minor difference is that we work in the Cartesian product $\re^{p+n}$. This is the step where the effort of choosing the normalization \eqref{equation:normalization:aug} in the augmented space is appreciated. 

Finally, the proof of \Cref{theorem:general} for the multivariate case will be established by a direct, column-wise, application of \Cref{theorem:general:vectors}. The details are given in \Cref{ss:details}. Some technical lemmas are proven in the Appendix. 

\bibliography{RE_condition}

\begin{thebibliography}{}

\bibitem[Balmand and Dalalyan, 2015]{balmand:dalalyan2015}
Balmand, S. and Dalalyan, A.~S. (2015).
\newblock Convex programming approach to robust estimation of a multivariate
  gaussian model.
\newblock {\em {\tt arXiv. 1512.04734}}.

\bibitem[Bickel et~al., 2009]{BRT}
Bickel, P.~J., Ritov, Y., and Tsybakov, A.~B. (2009).
\newblock Simultaneous analysis of {L}asso and {D}antzig selector.
\newblock {\em Ann. Statist.}, 37(4):1705--1732.

\bibitem[Boucheron et~al., 2013]{boucheron:lugosi:massart2013}
Boucheron, S., Lugosi, G., and Massart, P. (2013).
\newblock {\em Concentration inequalities: a nonasymptotic theory of
  independence}.
\newblock Oxford University Press.

\bibitem[Cand\`es, 2008]{candes2008}
Cand\`es, E. (2008).
\newblock The restricted isometry property and its implications for compressed
  sensing.
\newblock {\em C. R. Math. Acad. Sci. Paris}, 346(9-10):589--592.

\bibitem[Cand\`es and Randall, 2008]{candes:randall2008}
Cand\`es, E. and Randall, P.~A. (2008).
\newblock Highly robust error correction by convex programming.
\newblock {\em IEEE Trans. Inform. Theory}, 54(7):2829--2840.

\bibitem[Cand\`es and Recht, 2009]{candes:recht2009}
Cand\`es, E. and Recht, B. (2009).
\newblock Exact matrix completion via convex optimization.
\newblock {\em Foundations of Computational Mathematics}, 9:717--772.

\bibitem[Cand\`es and Tao, 2007]{candes:tao2007}
Cand\`es, E. and Tao, T. (2007).
\newblock The {D}antzig selector: statistical estimation when p is much larger
  than n.
\newblock {\em Ann. Statist.}, 35(6):2313--2351.

\bibitem[Chatterjee, 2005]{chatterjee}
Chatterjee, S. (2005).
\newblock An error bound in the sudakov–fernique inequality.
\newblock {\em {\tt arXiv. 0510424}}.

\bibitem[Dalalyan and Keriven, 2009]{dalalyan:keriven2009}
Dalalyan, A. and Keriven, R. (2009).
\newblock L1-penalized robust estimation for a class of inverse problems
  arising in multiview geometry.
\newblock In {\em Advances in Neural Information Processing Systems 22}.
  http://nips.cc/.

\bibitem[Dalalyan and Keriven, 2012a]{dalalyan:chen2012}
Dalalyan, A. and Keriven, R. (2012a).
\newblock Fused sparsity and robust estimation for linear models with unknown
  variance.
\newblock In {\em Advances in Neural Information Processing Systems 25}.
  http://nips.cc/.

\bibitem[Dalalyan and Keriven, 2012b]{dalalyan:keriven2012}
Dalalyan, A. and Keriven, R. (2012b).
\newblock Robust estimation for an inverse problem arising in multiview
  geometry.
\newblock {\em J. Math. Imaging Vision}, 43(1):10--23.

\bibitem[Fernique, 1975]{fernique1975}
Fernique, X. (1975).
\newblock {\em Regularit\'e des trajectoires des fonctions al'eatoires
  Gaussiens}, volume 480 of {\em Lecture Notes in Mathematics}.
\newblock Springer.

\bibitem[Huber and Ronchetti, 2009]{huber:ronchetti2009}
Huber, P.~J. and Ronchetti, E.~M. (2009).
\newblock {\em Robust statistics}.
\newblock Wiley Series in Probability and Statistics. John Wiley \& Sons, Inc.,
  Hoboken, NJ, second edition.

\bibitem[Ledoux and Talagrand, 1991]{ledoux:talagrand1991}
Ledoux, M. and Talagrand, M. (1991).
\newblock {\em Probability in Banach spaces}.
\newblock Springer, Berlin, Heidelberg.

\bibitem[Loh and Tan, 2018]{Loh2018}
Loh, P.-L. and Tan, X.~L. (2018).
\newblock High-dimensional robust precision matrix estimation: Cellwise
  corruption under $\epsilon$-contamination.
\newblock {\em Electron. J. Statist.}, 12(1):1429--1467.

\bibitem[Meinshausen and Bühlmann, 2006]{meinshausen:buhlmann2006}
Meinshausen, N. and Bühlmann, P. (2006).
\newblock High-dimensional graphs and variable selection with the lasso.
\newblock {\em Ann. Statist.}, 34(3):1436--1462.

\bibitem[Nguyen and Tran, 2013]{nguyen:tran2013}
Nguyen, N.~H. and Tran, T.~D. (2013).
\newblock Robust lasso with missing and grossly corrupted observations.
\newblock {\em IEEE Trans. Inform. Theory}, 59(4):2036--2058.

\bibitem[Oliveira, 2013]{oliveira2013}
Oliveira, R. (2013).
\newblock The lower tail of random quadratic forms, with applications to
  ordinary least squares and restricted eigenvalue properties.
\newblock {\em {\tt arXiv. 1312.2903}}.

\bibitem[Oliveira, 2016]{oliveira2016}
Oliveira, R. (2016).
\newblock The lower tail of random quadratic forms with applications to
  ordinary least squares.
\newblock {\em Probability Theory and Related Fields}, 166(3-4):1175--1194.

\bibitem[Raskutti et~al., 2010]{raskutti:wainwright:yu2010}
Raskutti, G., Wainwright, M.~J., and Yu, B. (2010).
\newblock {Restricted eigenvalue properties for correlated Gaussian designs}.
\newblock {\em J. Mach. Learn. Res.}, 11:2241--2259.

\bibitem[Rudelson and Zhou, 2013]{rudelson:zhou2013}
Rudelson, M. and Zhou, S. (2013).
\newblock Reconstruction from anisotropic random measurements.
\newblock {\em IEEE Trans. Inf. Theory}, 59(6):3434--3447.

\bibitem[Sudakov, 1971]{sudakov1971}
Sudakov, V. (1971).
\newblock Gaussian random processes and measures of solid angles in hilbert
  space.
\newblock {\em Dokl. Akad. Nauk SSSR}, 197(1):43--45.

\bibitem[Tarr et~al., 2016]{tarr:muller:weber2016}
Tarr, G., Muller, S., and Weber, N.~C. (2016).
\newblock Robust estimation of precision matrices under cellwise contamination.
\newblock {\em Computational Statistics {\&} Data Analysis}, 93:404--420.

\bibitem[Tibshirani, 1996]{tibshirani1996}
Tibshirani, R. (1996).
\newblock Regression shrinkage and selection via the {L}asso.
\newblock {\em Journal of the Royal Statistical Society. Series B},
  58(1):267--288.

\end{thebibliography}

%
%
\onecolumn
\section{Proof of Theorems \ref{theorem:general:vectors} and \ref{theorem:general}}\label{s:proof}
We set some notations. By $\mathbb{S}^{k-1}$ we denote the unit sphere in $\re^k$  and, with a slight abuse of notation, $\re^k$ will be identified with $\re^{k\times1}$.

We start by specifying a simplified setting, which will be extended later. 
\begin{itemize}
\item[(i)] We will first obtain the lower bound of \Cref{theorem:general:vectors} 
over vectors $\bv\in\re^{p+n}$. The lower bound for matrices in terms of mixed-norms 
will be obtained from the vector case applied to each column of the matrix.
\item[(ii)] We will first obtain a restricted version of \Cref{theorem:general:vectors} 
over vectors satisfying $\bv\in V(r_1,r_2)$. Here, $r_1,r_2>0$ are arbitrary and 
$V(r_1,r_2)$ is the compact subset of vectors $\bv = [\bb;\btheta]\in\re^{p+n}$, 
defined in \eqref{equation:V(r1,r2)}, satisfying the $\ell_2$-norm constraint 
$\Vert{\bfSigma}_S^{1/2}\bb\Vert_2^2+\Vert \btheta\Vert_2^2=1$ and the $\ell_1$-norm inequality constraints $\Vert \bb\Vert_1\le r_1$ and $\Vert \btheta\Vert_1\le r_2$. 
The obtain the general statement will be lifted from $V(r_1,r_2)$ to $\re^{p+n}$ 
via a standard peeling argument and by invoking the homogeneity of the norm in 
 $\re^{p+n}$.
\end{itemize}
We define, for all $r_1,r_2>0$, the set
\begin{align}
V(r_1,r_2):=\left\{
[
\bb;
\btheta
]
\in\re^{(p+n)\times1}:\Vert{\bfSigma}_S^{1/2}\bb\Vert_2^2+\Vert\btheta\Vert_2^2=1, \Vert \bb\Vert_1\le r_1,\Vert\btheta\Vert_1\le r_2\right\},
\label{equation:V(r1,r2)}
\end{align}
and the sets
\begin{align}
V_1(r_1)&:=\{\bb\in\re^p:\Vert {\bfSigma}_S^{1/2}\bb\Vert_2=1,\Vert \bb\Vert_1\le r_1\},\\
\overline V_1(r_1)&:=\{\bb\in\re^p:\Vert{\bfSigma}_S^{1/2}\bb\Vert_2\le1,\Vert \bb\Vert_1\le r_1\},\\
\overline V_2(r_2)&:=\{\btheta\in\re^n:\Vert \btheta\Vert_2\le1,\Vert \btheta\Vert_1\le r_2\}.
\end{align}


Set $\bfM = \left[\bfX^{(n)}\,|\; -\bfI_n\,\right]$. Taking (i)-(ii) for granted, a 
first step in order to prove the lower bound of 
\Cref{theorem:general:vectors} is to show that the restricted random variable
$$
-\inf_{\bv\in 
V(r_1,r_2)}
\Vert\bfM \bv\Vert_2
=
\sup_{\bv\in V(r_1,r_2)}-\Vert\bfM \bv\Vert_2
=
\sup_{\bv\in V(r_1,r_2)}
\inf_{\bu\in\mathbb{S}^{n-1}}\bu^\top(-\bfM)\bv,
$$
is, with high probability, upper bounded by a negative number as long as ${r_1^2}/{n}$ 
and ${r_2^2}/{n}$ are sufficiently small. For that purpose, we shall define the 
augmented matrix 
$\mathbf{K}\in\re^{n\times (p+n)}$ by
$$
\mathbf{K}=\left[\bfX_{\mathcal{R}}\big|-\sqrt{n}\bfI_{n}\right].
$$
Note that, for every 
$\bv=[\bb;\btheta]\in V(r_1,r_2)$, we have 
$\Vert{\bfSigma}_S^{1/2}\bb\Vert_2\le1$ so that
$$
-\Vert(\sqrt{n}\bfM)\bv\Vert_2=-\left\Vert\mathbf{K}\bv+
\bfE_{\mathcal{D}}\bb\right\Vert_2
\le-\Vert\mathbf{K}\bv\Vert_2+\Vert\bfE_{\mathcal{D}}\bb\Vert_2
\le-\Vert\mathbf{K}\bv\Vert_2+\|\bfE_{\mathcal{D}}{\bfSigma}_S^{-1/2}\|.
$$
Hence, if we define
\begin{align}
M(r_1,r_2,\mathbf{K})
&:=\sup_{\bv\in V(r_1,r_2)}\inf_{\bu\in\mathbb{S}^{n-1}}\bu^\top(-\mathbf{K}^{(n)})\bv \\
&=\sup_{\bv\in V(r_1,r_2)}-\left\Vert\left(\bfY^{(n)}+\bfE^{(n)}_{\mathcal{R}}\right)\bb-\btheta\right\Vert_2,
\label{equation:def:M}
\end{align}
we have the following bound:
\begin{align}
\sup_{\bv\in V(r_1,r_2)}-\Vert\bfM \bv\Vert_2
\le M(r_1,r_2,\mathbf{K})+\|\bfE_{\mathcal{D}}^{(n)}{\bfSigma}_S^{-1/2}\|.
\label{equation:uncontamination:bound}
\end{align} 
We thus aim to obtain an upper bound for $M(r_1,r_2,\mathbf{K})$ with high probability.  

\subsection{Bound in expectation}\label{ss:expectation}

The first key step is to obtain an upper bound in expectation. This will be the most delicate part of the proof. 

\begin{lemma}[Bound in expectation via a splitting argument]\label{lemma:expectation}
For any $n\ge10$, $\epsilon,\alpha,\beta>0$ and $r_1,r_2>0$ for which $V(r_1,r_2)$ is nonempty, 
$$
\esp\left[M(r_1,r_2,\mathbf{K})\right]\le
-\frac{(3/4-\epsilon)}{\sqrt{2}}\left(1-e^{-\frac{n\epsilon^2}{2}}\right)+\frac{1}{2}\left(\frac{1}{\alpha}+\frac{1}{\beta}\right)+\left(2+\alpha\sqrt{2}\right)r_1\varrho({\bfSigma}_S)\sqrt{\frac{\log p}{n}}+\beta r_2\sqrt{\frac{2\log n}{n}}.
$$ 
\end{lemma}

To prove the previous lemma we shall need the two following intermediate results whose proofs are postponed to the Appendix for increased readability.
\begin{lemma}[A lower bound in high-probability for standard Gaussian ensembles]\label{lemma:aux1}
For all $n\ge10$, $t>0$ and $r_1>0$ for which $V_1(r_1)$ is nonempty, with probability at least $1-\exp(-nt^2/2)$,
$$
\inf_{\bb\in V_1(r_1)}\left\Vert\bfX_{\mathcal{R}}^{(n)}\bb\right\Vert_2\ge\frac{3}{4}-t-r_1\varrho({\bfSigma}_S)\sqrt{\frac{2\log p}{n}}.
$$
\end{lemma}

\begin{lemma}[An upper bound in expectation for corrupted Gaussian designs]\label{lemma:aux2}
For all $r_2,r_2>0$ for which $V(r_1,r_2)$ is nonempty,
$$
\esp\left[\sup_{[\bb;\btheta]\in V(r_1,r_2)}\bb^\top(\bfX_{\mathcal{R}}^{(n)})^\top\btheta\right]\le r_1\varrho({\bfSigma}_S)\sqrt{\frac{2\log p}{n}}+r_2\sqrt{\frac{2\log n}{n}}. 
$$
\end{lemma}
 
\begin{proof}[Proof of Lemma \ref{lemma:expectation}]
In the sequel, we set $M:=M(r_1,r_2,\mathbf{K})$ for simplicity of notation. We also 
note that $\mathbf{\widetilde X}_{\mathcal{R}}=\bfX_{\mathcal{R}}{\bfSigma}_S^{-1/2}$ 
is a standard Gaussian ensemble. Indeed, for any $i\in [n]$, $(\bfX_{\mathcal{R}})_{i,\bullet}=\bfY_{i,\bullet}+(\bfE_{\mathcal{R}})_{i,\bullet}$, is a sum of two independent centered Gaussian vectors. Hence, the rows of $\bfX_{\mathcal{R}}$ are i.i.d. centered Gaussian vectors.

Let  $[\bb;\btheta]$ be any point from $V(r_1,r_2)$. 
Since $\Vert{\bfSigma}_S^{1/2} \bb\Vert^2+\Vert \btheta\Vert^2=1$, we necessarily have 
one of the alternatives:
\begin{align}
\begin{cases}
\mbox{(A)}\quad\Vert\btheta\Vert_2^2\le\frac{1}{2}\mbox{ and }\Vert {\bfSigma}_S^{1/2}\bb\Vert_2^2\ge\frac{1}{2}\\
\mbox{(B)}\quad\Vert\btheta\Vert_2^2\ge\frac{1}{2}\mbox{ and }\Vert {\bfSigma}_S^{1/2}\bb\Vert_2^2\le\frac{1}{2}.
\end{cases}
\end{align} 
In case (B) holds, $\Vert\bfX_{\mathcal{R}}^{(n)}\bb\Vert_2^2+\Vert\btheta\Vert_2^2\ge1/2$. On the other hand, in case (A) holds, we have 
\begin{align}
\sqrt{\Vert\bfX^{(n)}_{\mathcal{R}}\bb\Vert_2^2+\Vert\btheta\Vert_2^2}
&\ge \Vert\bfX_{\mathcal{R}}^{(n)}\bb\Vert_2 \\
&\ge \inf\left\{\Vert\bfX_{\mathcal{R}}^{(n)}\bb'\Vert_2:\frac{1}{2}\le\Vert {\bfSigma}_S^{1/2}\bb'\Vert_2^2\le1,\Vert \bb'\Vert_1\le r_1\right\} \\
&=\inf\left\{\left\Vert\mathbf{\widetilde X}_{\mathcal{R}}^{(n)}\left(\frac{{\bfSigma}_S^{1/2}\bb'}{\Vert{\bfSigma}_S^{1/2}\bb'\Vert_2}\right)\right\Vert_2\Vert {\bfSigma}_S^{1/2}\bb'\Vert_2:\frac{1}{2}\le\Vert{\bfSigma}_S^{1/2} \bb'\Vert_2^2\le1,\Vert \bb'\Vert_1\le r_1\right\} \\
&\ge\frac{1}{\sqrt{2}}\inf\left\{\left\Vert\mathbf{\widetilde X}_{\mathcal{R}}^{(n)}\left(\frac{{\bfSigma}_S^{1/2}\bb'}{\Vert{\bfSigma}_S^{1/2}\bb'\Vert_2}\right)\right\Vert_2 :\frac{1}{2}\le\Vert{\bfSigma}_S^{1/2} \bb'\Vert_2^2\le1,\Vert \bb'\Vert_1\le r_1\right\} \\
&\ge \frac{1}{\sqrt{2}}\inf\left\{\left\Vert\mathbf{\widetilde X}_{\mathcal{R}}^{(n)}\mathbf{y}\right\Vert_2:\Vert\mathbf{y}\Vert_2=1,\left\Vert {\bfSigma}_S^{-1/2}\mathbf{y}\right\Vert_1\le \sqrt{2}r_1\right\}\\
&=\frac{1}{\sqrt{2}}\inf_{\bb\in V_1(\sqrt{2}r_1)}\left\Vert\mathbf{X}_{\mathcal{R}}^{(n)}\bb\right\Vert_2.
\end{align}
In the forth inequality above, we have used the fact that 
\begin{align}
\left\{\frac{{\bfSigma}_S^{1/2}\bb'}{\left\Vert{\bfSigma}_S^{1/2}\bb'\right\Vert_2}:\frac{1}{2}\le\Vert{\bfSigma}_S^{1/2} \bb'\Vert_2^2\le1,\Vert \bb'\Vert_1\le r_1\right\}\subset {\bfSigma}_S^{1/2}\cdot V_1\left(\sqrt{2}r_1\right),\label{lemma:expectation:eq0}
\end{align}
where 
$
{\bfSigma}_S^{1/2}\cdot V_1(\sqrt{2}r_1):=
\{\mathbf{y}\in\re^p:\Vert\mathbf{y}\Vert_2=1,\Vert {\bfSigma}_S^{-1/2}\mathbf{y}\Vert_1\le \sqrt{2}r_1\}.
$
Indeed, given $\mathbf{y}:=\frac{{\bfSigma}_S^{1/2}\bb'}{\Vert{\bfSigma}_S^{1/2}\bb'\Vert_2}$ with $\bb'$ belonging to the set in the left hand side of \eqref{lemma:expectation:eq0}, we have that $\Vert\mathbf{y}\Vert_2=1$ and
$$
r_1\ge \Vert \bb'\Vert_1 =\left\Vert {\bfSigma}_S^{-1/2}\mathbf{y}\right\Vert_1\cdot\left\Vert {\bfSigma}_S^{1/2}\bb'\right\Vert_2 \ge\frac{\left\Vert {\bfSigma}_S^{-1/2}\mathbf{y}\right\Vert_1}{\sqrt{2}}. 
$$

From the previous conclusions of conditions (A) or (B), we obtain that \emph{for all} $(\bb^\top,\btheta^\top)^\top\in V(r_1,r_2)$,
\begin{align}
\Vert\bfX_{\mathcal{R}}^{(n)}\bb\Vert_2^2+\Vert\btheta\Vert_2^2\ge
\frac{1}{2}\min\left\{\inf_{\bb'\in V_1(\sqrt{2}r_1)}\Vert\mathbf{X}_{\mathcal{R}}^{(n)}\bb'\Vert_2^2,1\right\}.
\label{lemma:expectation:eq1}
\end{align}
From \eqref{lemma:expectation:eq1}, we obtain that 
\begin{align}
M^2&=\inf_{[\bb;\btheta]\in V(r_1,r_2)}
\Vert\bfX_{\mathcal{R}}^{(n)}\bb-\btheta\Vert_2^2\\
&=\inf_{[\bb;\btheta]\in V(r_1,r_2)}\left[\Vert\bfX_{\mathcal{R}}^{(n)}\bb\Vert_2^2
+\Vert\btheta\Vert_2^2-2\bb^\top(\bfX_{\mathcal{R}}^{(n)})^\top\btheta\right]\\
&\ge \frac{1}{2}\min\bigg\{\inf_{\bb'\in V_1(\sqrt{2}r_1)}\Vert\mathbf{X}_{\mathcal{R}}^{(n)}\bb'\Vert_2^2\,,1\bigg\}
-\Big|\sup_{[\bb;\btheta]\in V(r_1,r_2)}2\bb^\top(\bfX_{\mathcal{R}}^{(n)})^\top\btheta
\Big|.
\end{align}
After rearranging and taking the square root, we obtain
\begin{align}
M=-|M|\le -\frac{1}{\sqrt{2}}\min\bigg\{\inf_{\bb\in V_1(\sqrt{2}r_1)}\Vert\mathbf{X}_{\mathcal{R}}^{(n)}\bb\Vert_2\,,1\bigg\}+\Big(2\sup_{
[\bb;\btheta]\in V(r_1,r_2)}\bb^\top(\bfX_{\mathcal{R}}^{(n)})^\top\btheta\Big)^{1/2}.
\label{lemma:expectation:eq2}
\end{align}
It follows from the above expression, that we need to lower bound the mapping
$V_1\left(\sqrt{2}r_1\right)\ni \bb\mapsto\Vert\mathbf{X}_{\mathcal{R}}^{(n)}\bb\Vert_2,$
depending solely on the ``pure'' parameter-vector $\bb$ and the Gaussian design 
$\mathbf{X}_{\mathcal{R}}$ with i.i.d. $\mathcal{N}_p(0,{\bfSigma}_S)$ rows. Moreover, it is sufficient to obtain an upper bound on the bilinear form
$
V(r_1,r_2)\ni(\bb,\btheta)\mapsto \btheta^\top\bfX_{\mathcal{R}}^{(n)}\bb,
$
which depends on the interaction between $\bb$ and $\btheta$ as well as on 
$\bfE_{\mathcal{R}}$. 

We first use Lemma \ref{lemma:aux2} to bound the expectation of the second term in the RHS of \eqref{lemma:expectation:eq2}. This and Jensen's inequality imply that
\begin{align}
\esp\left[\Big(2\sup_{
[\bb;\btheta]\in V(r_1,r_2)}\bb^\top(\bfX_{\mathcal{R}}^{(n)})^\top\btheta\Big)^{1/2}
\right]
&\le \esp\left[2\sup_{[\bb;\btheta]\in V(r_1,r_2)}\bb^\top(\bfX_{\mathcal{R}}^{(n)})^\top\btheta\right]^{1/2} \\
&\le \bigg\{2\varrho({\bfSigma}_S)r_1\sqrt{\frac{2\log p}{n}}+2r_2\sqrt{\frac{2\log n}{n}}\bigg\}^{1/2} \\
&\le \bigg\{\varrho({\bfSigma}_S)r_1\sqrt{\frac{8\log p}{n}}\bigg\}^{1/2}
+\bigg\{r_2\sqrt{\frac{8\log n}{n}}\bigg\}^{1/2} \\
&\le \frac{\alpha\varrho({\bfSigma}_S)r_1}{2}\sqrt{\frac{8\log p}{n}}+
\frac{1}{2\alpha}
+\frac{\beta r_2}{2}\sqrt{\frac{8\log n}{n}}+\frac{1}{2\beta},
\label{lemma:expectation:eq3}
\end{align}
for all $\alpha,\beta>0$. In above, we have used the well known bound $2\sqrt x
\le\lambda x+\lambda^{-1}$ with $x:=\frac{8}{n}\varrho({\bfSigma}_S)^2r_1^2\log p$, 
$\lambda:=\alpha$ and with $x:={\frac{8}{n}r_2^2\log n}$, $\lambda:=\beta$.

We now give a lower bound estimate for the expectation of the first term in the RHS of \eqref{lemma:expectation:eq2}. For all $\epsilon>0$, we define the event 
$$
A_{\epsilon}(r_1):=\left\{\inf_{\bb\in V_1(\sqrt{2}r_1)}\Vert\mathbf{X}_{\mathcal{R}}^{(n)}\bb\Vert_2\ge3/4-\epsilon-(\sqrt{2}r_1)\varrho({\bfSigma}_S)\sqrt{\frac{2\log p}{n}}\right\},
$$
so that
\begin{align}
\mathsf{1}_{A_\epsilon}(r_1)\min\left\{\inf_{\bb\in V_1(\sqrt{2}r_1)}\Vert\mathbf{X}_{\mathcal{R}}^{(n)}\bb\Vert_2,1\right\}&\ge \mathsf{1}_{A_\epsilon(r_1)}\left(3/4-\epsilon-2r_1\varrho({\bfSigma}_S)\sqrt{\frac{\log p}{n}}\right),\label{lemma:expectation:eq3'}\\
\mathsf{1}_{A_\epsilon^c(r_1)}\min\left\{\inf_{\bb\in V_1(\sqrt{2}r_1)}\Vert\mathbf{X}_{\mathcal{R}}^{(n)}\bb\Vert_2,1\right\}&\ge 0.\label{lemma:expectation:eq3''}
\end{align}

From \eqref{lemma:expectation:eq2}-\eqref{lemma:expectation:eq3''} and the set partition $\mathsf{1}=\mathsf{1}_{A_\epsilon}(r_1)+\mathsf{1}_{A_\epsilon^c}(r_1)$, we obtain that, for all $\epsilon,\alpha,\beta>0$,
\begin{align}
\esp[M]&\le -\frac{(3/4-\epsilon)}{\sqrt{2}}\prob(A_\epsilon(r_1))+\frac{1}{2}(\alpha^{-1}+\beta^{-1}) \\
&+(2+\sqrt{2}\alpha)r_1\varrho({\bfSigma}_S)\sqrt{\frac{\log p}{n}}+\beta r_2\sqrt{\frac{2\log n}{n}}.
\label{lemma:expectation:eq4}
\end{align}
From Lemma \ref{lemma:aux1}, we can control the above expectation since $\prob(A_\epsilon(r_1))=1-e^{-n\epsilon^2/2}$. This and \eqref{lemma:expectation:eq4} finish the proof. 
\end{proof}

\subsection{Gaussian concentration inequality}\label{ss:concentration}
From Lemma \ref{lemma:expectation}, we obtain that, for any $n\ge10$ and $\epsilon,\alpha,\beta>0$, 
\begin{align}
\esp\left[M(r_1,r_2,\mathbf{K})\right]+1-\left(\frac{1}{2\alpha}+\frac{1}{2\beta}\right)\le t_{\epsilon}(r_1,r_2),
\label{equation:M:centering}
\end{align}
where we have defined the quantity
\begin{align}
t_{\epsilon}(r_1,r_2)&:=1-\frac{(3/4-\epsilon)}{\sqrt{2}}\left(1-e^{-\frac{n\epsilon^2}{2}}\right)+\left(2+\alpha\sqrt{2}\right)r_1\varrho({\bfSigma}_S)\sqrt{\frac{\log p}{n}}+\beta r_2\sqrt{\frac{2\log n}{n}}.
\label{equation:t(r1,r2)}	
\end{align}

After obtaining a control of the expectation, we shall now obtain a control on the variance of the random variable $M(r_1,r_2,\mathbf{K})$ (defined as an extremum of an
empirical process) in order to obtain an upper tail inequality for $M(r_1,r_2,\mathbf{K})-\esp\left[M(r_1,r_2,\mathbf{K})\right]$. In that respect, we make three important observations in our context where contamination is present. 
\begin{itemize}
\item[(i)] The corruption vector $\btheta$ acts as a ``bias'' in the empirical process of \eqref{equation:def:M}. Hence, in order to control the variance, the constraint over $\btheta$ has no significant impact.
\item[(ii)] We will crucially use that $\bfY{\perp\!\!\!\perp}\bfE_{\mathcal{R}}$.
\item[(iii)] Taking (ii) for granted, we will control $\bb\mapsto\mathbf
{X}_{\mathcal{R}}\bb$ using that $\bfX_{\mathcal{R}}$ is a Gaussian design (hence, we may use the Gaussian concentration inequality).
\end{itemize}

\begin{lemma}[Concentration around the mean]\label{lemma:concentration:mean}
For any $n\ge10$, $\epsilon\in(0,3/4)$, $\alpha,\beta,\sigma>0$ and $r_1,r_2>0$ for which $V(r_1,r_2)$ is nonempty,
\begin{align}
\prob\left\{M(r_1,r_2,\mathbf{K})+1\ge\left(1+\frac{1}{\sigma}\right)t_{\epsilon}(r_1,r_2)\right\}\le\exp\left[-\frac{t_{\epsilon}(r_1,r_2)^2}{2\sigma^2}n\right].
\end{align}
\end{lemma}
\begin{proof}
We recall that, since $\bfY{\perp\!\!\!\perp}\bfE_{\mathcal{R}}$, $\bfX_{\mathcal{R}}=\bfY+\bfE_{\mathcal{R}}$ is a Gaussian ensemble with independent rows. In particular, we may write $\bfX_{\mathcal{R}}=\mathbf{\widetilde X}_{\mathcal{R}}{\bfSigma}_S^{1/2}$, where $\mathbf{\widetilde X}_{\mathcal{R}}\in\re^{n\times p}$ is a standard Gaussian ensemble.

Let $\bfY,\bfY'\in\re^{n\times p}$ and $\bfE_{\mathcal{R}},\bfE_{\mathcal{R}}'\in\re^{n\times p}$. Set $\bfX_{\mathcal{R}}:=\bfY+\bfE_{\mathcal{R}}$, $\bfX_{\mathcal{R}}'=\bfY'+\bfE_{\mathcal{R}}'$ as well as
$
\mathbf{K}=\left[\bfY+\bfE_{\mathcal{R}}\big|-\sqrt{n}\bfI_{n}\right]
$ 
and
$
\mathbf{K'}=\left[\bfY'+\bfE_{\mathcal{R}}'\big|-\sqrt{n}\bfI_{n}\right].
$
We have
\begin{align}
M(r_1,r_2,\mathbf{K})-M(r_1,r_2,\mathbf{K}')&=
\sup_{\bv\in V(r_1,r_2)}-\left\Vert n^{-1/2}\bfX_{\mathcal{R}}\bb-\btheta
\right\Vert_2
-\sup_{\bv\in V(r_1,r_2)}-\left\Vert n^{-1/2}\bfX_{\mathcal{R}}'b-\btheta
\right\Vert_2\\
&\le \sup_{\bv\in V(r_1,r_2)}\left[
-\left\Vert n^{-1/2}\bfX_{\mathcal{R}}\bb-\btheta
\right\Vert_2
+\left\Vert n^{-1/2}\bfX_{\mathcal{R}}'b-\btheta
\right\Vert_2
\right]\\
&\le n^{-1/2}\sup_{\bb:\|\bfSigma^{1/2}\bb\|_2\le 1}
\left\Vert \mathbf{\widetilde X}_{\mathcal{R}}'{\bfSigma}_S^{1/2}\bb-\mathbf{\widetilde X}_{\mathcal{R}}{\bfSigma}_S^{1/2}\bb\right\Vert_2\\
&\le n^{-1/2}\sup_{\bb:\|\bfSigma^{1/2}\bb\|_2\le 1}\left\Vert{\bfSigma}_S^{1/2}\bb\right\Vert_2\left\Vert\mathbf{\widetilde X}_{\mathcal{R}}-\mathbf{\widetilde X}_{\mathcal{R}}'\right\Vert\\
&\le n^{-1/2}\left\Vert\mathbf{\widetilde X}_{\mathcal{R}}-\mathbf{\widetilde X}_{\mathcal{R}}'\right\Vert_{2,2},
\end{align}
and similarly for $M(r_1,r_2,\mathbf{K}')-M(r_1,r_2,\mathbf{K})$. 

We thus conclude that $\mathbf{\widetilde X}_{\mathcal{R}}\mapsto M(r_1,r_2,\mathbf{K})$ is a $n^{-1/2}$-Lipschitz function. From this and the fact that $\mathbf{\widetilde X}_{\mathcal{R}}\in\re^{n\times p}$ is a standard Gaussian ensemble, we obtain from Theorem 5.6 in \citep{boucheron:lugosi:massart2013} that, for all $t>0$,
\begin{align}
\prob\left[M(r_1,r_2,\mathbf{K})-\esp[M(r_1,r_2,\mathbf{K})]\ge t\right]\le\exp\left(-\frac{nt^2}{2}\right).
\end{align}
Using \eqref{equation:M:centering} and letting $t:=\frac{1}{\sigma}t_{\epsilon}(r_1,r_2)$ above, we prove the claim.
\end{proof}

\subsection{Lifting to the augmented Euclidean space $\re^{p+n}$}\label{ss:peeling}
We now aim in removing the constraints $\Vert \bb\Vert_1\le r_1$ and $\Vert\btheta\Vert_1\le r_2$ by using a standard peeling argument. In our case, we need an ``augmented version'' which is easy to generalize.
\begin{lemma}[A augmented peeling argument]\label{lemma:peeling}
Suppose $g:\re^2\rightarrow\re$ is a nonnegative strictly increasing function for which $g(r_1,r_2)\ge\mu>0$ for all $r_1,r_2>0$. Suppose that $h_1:\re^p\rightarrow\re$ and $h_2:\re^n\rightarrow\re$ are nonnegative increasing functions, $A\subset\re^{n\times p}$ is a nonempty set and $h(\bb,\btheta):=(h_1(\bb),h_2(\btheta))$. Suppose further that $f(\cdot;X):\re^{n\times p}\rightarrow\re$ is random function dependent on some random vector $X$ such that, for some\footnote{Here $h(\bb,\btheta):=(h_1(\bb),h_2(\btheta))\le(r_1,r_2)$ means ``$h_1(\bb)\le r_1$ and $h_2(\btheta)\le r_2$".} $c>0$,
$$
\prob\left\{
\sup_{
\bv\in A,h(\bv)\le(r_1,r_2)
} 
f(\bv;X)>g(r_1,r_2)
\right\}\le\exp\left[-cg^2(r_1,r_2)\right].
$$
Then, for any $\tau>0$,
$$
\prob\left[\exists \bv\in A \mbox{ such that } f(\bv;X)\ge(1+\tau)g(h(\bv))\right]
\le \frac{\exp\left[-(1+\tau)^2c\mu^2\right]}{1-\exp\left[-(1+\tau)^2c\mu^2\right]}.
$$
\end{lemma}
\begin{proof}
The proof follows closely the proof of Lemma 3 in [5] with two almost immediate changes. First, we note that the same argument holds true for the Cartesian product of the sets 
$\{\bv=(\bb,\btheta)\in A:h_1(\bb)\le r_1\}$ and $\{\bv=(\bb,\btheta)\in A:h_2(\btheta)\le r_2\}$. Second, the argument still holds true with the factor $2$ replaced by\footnote{This allows a better bias-variance trade-off in order to obtain a sharper restricted eigenvalue constant of Theorem 1.} $1+\tau$.
\end{proof}

%
%

By recalling the definition of $V(r_1,r_2)$ in \eqref{equation:V(r1,r2)}, we now use the consecutive upper bounds in \eqref{equation:uncontamination:bound} and Lemma \ref{lemma:concentration:mean}, \eqref{equation:t(r1,r2)} and Lemma \ref{lemma:peeling} with the following setup: for $\epsilon\in(0,3/4)$, $\alpha,\beta,\sigma>0$, we take 
\begin{align}
h_1(\bb)&:=\Vert \bb\Vert_1,\\
h_2(\btheta)&:=\Vert\btheta\Vert_1,\\
A &:=\left\{(\bb,\btheta)\in\re^p\times\re^n:\Vert{\bfSigma}_S^{1/2}\bb\Vert_2^2+\Vert\btheta\Vert_2^2=1\right\},\\
f(v;\bfX)&:=-\left\Vert\mathbf{K}\bv\right\Vert_2+1+\left(\frac{1}{2\alpha}+\frac{1}{2\beta}\right)-\frac{\sigma_{\max}(\bfE_{\mathcal{D}}{\bfSigma}_S^{-1/2})}{\sqrt{n}},\\
g(r_1,r_2)&:=(1+1/\sigma)t_{\epsilon}(r_1,r_2),
\end{align}
and
$$
\mu:=\left(1+1/\sigma\right)\left[1-\frac{(3/4-\epsilon)}{\sqrt{2}}\right],\quad\quad
c:=\frac{n}{2(1+\sigma)^2}.
$$
For simplicity, we recall the definitions
\begin{align}
\rho:=(1+\tau)(1+1/\sigma),
\label{equation:def:rho}
\end{align}
and, for any matrix $\bfA \in\re^{n\times p}$, 
\begin{align}
\mathsf{C}_{n}(\bfA )&:=\rho\left[\frac{(3/4-\epsilon)}{\sqrt{2}}
(1-e^{-\frac{n\epsilon^2}{2}})-(1-\rho^{-1})\right]\nonumber\\
&-\left(\frac{1}{2\alpha}+\frac{1}{2\beta}+
\frac{\sigma_{\max}(\bfA )}{\sqrt{n}}\right).
\end{align}

We thus obtain from Lemma \ref{lemma:peeling} that, for any $n\ge10$, $\epsilon\in(0,3/4)$ and $\alpha,\beta,\sigma,\tau>0$, with probability at least 
\begin{align}
1-\frac{\exp\left[-\frac{(1+\tau)^2\mu_{\epsilon}^2}{2\sigma^2}n\right]}{1-\exp\left[-\frac{(1+\tau)^2\mu_{\epsilon}^2}{2\sigma^2}n\right]},
\label{equation:probability}
\end{align}
the following bound holds: for all $v=(\bb,\btheta)\in A$,  
\begin{align}
-\left\Vert\bfM \bv\right\Vert_2+1+\left(\frac{1}{2\alpha}+\frac{1}{2\beta}\right)-\frac{\sigma_{\max}(\bfE_{\mathcal{D}}{\bfSigma}_S^{-1/2})}{\sqrt{n}}\le \rho t_\epsilon\left(\Vert \bb\Vert_1,\Vert\btheta\Vert_1\right).
\label{equation:bound:peeled}
\end{align}

We also remark that the expression in \eqref{equation:probability} can be lower bounded by
$
1-2\exp\left[-\frac{(1+\tau)^2\mu_{\epsilon}^2}{2\sigma^2}n\right],
$
as long as
$$
n\ge\left(\frac{2\sigma^2\log 2}{(1+\tau)^2\mu_{\epsilon}^2}\right).	
$$
We shall use this fact, definition \eqref{equation:t(r1,r2)}, the fact that for all nonzero $\bv=[\bb;\btheta]\in\re^{p+n}$, 
$$
\frac{\left[
\begin{array}{c}
\bb;
\btheta
\end{array}
\right]}
{\sqrt{\Vert{\bfSigma}_S^{1/2}\bb\Vert_2^2+\Vert\btheta\Vert_2^2}}
\in A,
$$ 
relation \eqref{equation:bound:peeled} and \emph{homogeneity of the norm} in $\re^{p+n}$.

The facts listed in the previous paragraph imply the following property: for all $\epsilon\in(0,3/4)$, positive $\alpha$, $\beta$, $\sigma$ and $\tau$ and $n\ge\left(\frac{2\sigma^2\log 2}{(1+\tau)^2\mu_{\epsilon}^2}\right)\vee10$, with probability at least 
$
1-2\exp\left[-\frac{(1+\tau)^2\mu_{\epsilon}^2}{2\sigma^2}n\right],
$ 
for all $\bv=[\bb;\btheta]\in\re^{p+n}$, 
\begin{align}
\left\Vert\bfM \bv\right\Vert_2 &\ge  
\mathsf{C}_{n}(\bfE_{\mathcal{D}}{\bfSigma}_S^{-1/2})
\left\Vert\left[ {\bfSigma}_S^{1/2}\bb;\btheta\right]\right\Vert_2-\rho\left[(2+\alpha\sqrt{2})\varrho({\bfSigma}_S)\Vert \bb\Vert_1\sqrt{\frac{\log p}{n}}+\beta\Vert\btheta\Vert_1\sqrt{\frac{2\log n}{n}}\right].
\label{equation:bound:peeled:2}
\end{align}
In the above expressions, $\mathsf{C}_{n}(\bfE_{\mathcal{D}}{\bfSigma}_S^{-1/2})$ is defined in \eqref{equation:def:CA}.

\subsection{Final details: proof of Theorem \ref{theorem:general}}\label{ss:details}
The obtained bound \eqref{equation:bound:peeled:2} concludes the proof of Theorem \ref{theorem:general:vectors} for vectors $v\in\re^{p+n}$. We now conclude the proof of Theorem \ref{theorem:general} for matrix parameters. Note that
$$
\left\Vert [{\bfSigma}_S^{1/2}\bb;\btheta]\right\Vert_2
\ge\Vert{\bfSigma}_S^{1/2}\bb\Vert_2\vee\Vert\btheta\Vert_2.
$$
To ease notation, we write $\mathsf{C}_{n}$ 
instead of $\mathsf{C}_{n}(\bfE_\mathcal{D}{\bfSigma}_S^{-1/2})$. 
Under the conditions of Theorem \ref{theorem:general:vectors}, we obtain that, with probability at least 
$
1-2\exp\left[-\frac{(1+\tau)^2\mu_{\epsilon}^2}{2\sigma^2}n\right],
$
for every $\mathbf{V}=[\bfB;{\bfTheta}]$ and every $j\in[p]$, we have
\begin{align}
\Vert\bfM\mathbf{V}_{\bullet,j}\Vert_2
&\ge \mathsf{C}_{n}
\left\Vert
\begin{bmatrix}
{\bfSigma}^{1/2}_S\bfB_{\bullet,j}\\
{\bfTheta}_{\bullet,j}
\end{bmatrix}
\right\Vert_2
-\rho\left[(2+\alpha\sqrt{2})\varrho({\bfSigma}_S)\Vert\bfB_{\bullet,j}\Vert_1\sqrt{\frac{\log p}{n}}+\beta\Vert{\bfTheta}_{\bullet,j}\Vert_1\sqrt{\frac{2\log n}{n}}\right]\\
&\ge \mathsf{C}_{n}
\left(\Vert{\bfSigma}^{1/2}_S\bfB_{\bullet,j}\Vert_2
\bigvee\Vert{\bfTheta}_{\bullet,j}\Vert_2\right)
-\rho\left[(2+\alpha\sqrt{2})\varrho({\bfSigma}_S)\Vert\bfB_{\bullet,j}\Vert_1\sqrt{\frac{\log p}{n}}+\beta\Vert{\bfTheta}_{\bullet,j}\Vert_1\sqrt{\frac{2\log n}{n}}\right].
\end{align}
We claim that the above set of column-wise inequalities implies that, for all 
$\mathbf{V}\in\re^{p+n, n}$,
\begin{align}
\Vert\bfM\mathbf{V}\Vert_{2,2} 
&\ge {\mathsf{C}_{n}}
\left(\Vert{\bfSigma}^{1/2}_S\bfB\Vert_{2,2}
\bigvee\Vert{\bfTheta}\Vert_{2,2}\right) 
-\rho\left[(2+\alpha\sqrt{2})\varrho({\bfSigma}_S)\Vert\bfB\Vert_{1,1}\sqrt{\frac{\log p}{n}}+\beta\Vert{\bfTheta}\Vert_{2,1}\sqrt{\frac{2\log n}{n}}\right].
\label{equation:final:bound}
\end{align}
This will complete the proof of \Cref{theorem:general}.

Indeed, in the previous inequality, moving the negatives terms in the RHS to the LHS, squaring, summing over $j\in[p]$ and applying the Minkowski inequality, we get
\begin{align}
 \Vert\bfM\mathbf{V}\Vert_{2,2}&+\rho\left(2+\alpha\sqrt{2}\right)\varrho({\bfSigma}_S)
\sqrt{\frac{\log p}{n}}\bigg(\sum_{j=1}^p\Vert\bfB_{\bullet,j}\Vert_1^2\bigg)^{1/2}
+\rho\beta\sqrt{\frac{2\log n}{n}}\bigg(\sum_{j=1}^p
\Vert{\bfTheta}_{\bullet,j}\Vert_1^2\bigg)^{1/2}\\
&\ge \mathsf{C}_{n}\big(\Vert{\bfSigma}^{1/2}_S\bfB\Vert_{2,2}\vee
\Vert{\bfTheta}\Vert_{2,2}\big).
\end{align}
It is not difficult to see that the first sum in the LHS can be upper bounded by $\Vert\bfB\Vert_{1,1}^2$. Finally, the second sum in the LHS may be bounded as
follows
\begin{align}
\sum_{j=1}^p\Vert{\bfTheta}_{\bullet,j}\Vert_1^2 &=
\sum_{j=1}^p\bigg(\sum_{i=1}^n|{\bfTheta}_{ij}|\bigg)^2\\
&=\sum_{j=1}^p\sum_{i=1}^n\sum_{k=1}^n|{\bfTheta}_{ij}||{\bfTheta}_{kj}|\\
&=\sum_{i=1}^n\sum_{k=1}^n\bigg(\sum_{j=1}^p|{\bfTheta}_{ij}||{\bfTheta}_{kj}|\bigg)\\
&\le \sum_{i=1}^n\sum_{k=1}^n\bigg({\sum_{j=1}^p|{\bfTheta}_{ij}|^2}\bigg)^{1/2}
\bigg({\sum_{j=1}^p|{\bfTheta}_{kj}|^2}\bigg)^{1/2}\\
&=\sum_{i=1}^n\sum_{k=1}^n\Vert{\bfTheta}_{i,\bullet}\Vert_2\Vert{\bfTheta}_{k,\bullet}\Vert_2\\
&=\bigg(\sum_{i=1}^n\Vert{\bfTheta}_{i,\bullet}\Vert_2\bigg)^2=
\Vert{\bfTheta}\Vert_{2,1}^2.
\end{align}
We thus conclude \eqref{equation:final:bound}, by using the previous bounds. 
This completes the proof of \Cref{theorem:general}. 

\section{Appendix}

\begin{proof}[Proof of Lemma \ref{lemma:aux1}]
The proof follows essentially by Lemmas 1-2 in [5] using Gordon's inequality and the Gaussian concentration inequality. We just make some minor remarks regarding numerical constants. First, since we only need a one sided tail inequality, the probability of the event can be improved to $1-\exp(-nt^2/2)$ rather than $1-2\exp(-nt^2/2)$ as given in the mentioned article. Second, the constant 2 in the inequality above is a consequence of Theorem 2.5 in \citep{boucheron:lugosi:massart2013} (while in [5] the constant presented is $9$).
\end{proof}

For the proof of Lemma \ref{lemma:aux2}, we recall Slepian's expectation comparison inequality, also known as Sudakov-Fernique's inequality  
\citep{sudakov1971,fernique1975,chatterjee,ledoux:talagrand1991}.
\begin{thm}[Sudakov-Fernique's inequality]\label{thm:slepian}
Let $\{X_{\bu}\}_{U}$ and $Y=\{Y_{\bu}\}_{U}$ be continuous Gaussian processes such that,  for all $u,u'\in U$, $\esp[X_{\bu}]=\esp[Y_{\bu}]$ and
$$
\esp\left[(X_{\bu}-X_{\bu'})^2\right]\le\esp\left[(Y_{\bu}-Y_{\bu'})^2\right].
$$
Then
$$
\esp\left[\sup_{\bu\in U}X_{\bu}\right]\le\esp\left[\sup_{\bu\in U}Y_{\bu}\right].
$$
\end{thm}

%

\begin{proof}[Proof of Lemma \ref{lemma:aux2}]
Note that 
\begin{align}
\sup_{[\bb;\btheta]\in V(r_1,r_2)}\bb^\top(\bfX_{\mathcal{R}}^{(n)})^\top\btheta
\le\sup_{(\bb,\btheta)\in \overline V_1(r_1)\times \overline V_2(r_2)}\bb^\top(\bfX_{\mathcal{R}}^{(n)})^\top\btheta.
\label{lemma:aux2:eq0}
\end{align}

Since $\mathbf{Y}$ and $\mathbf{E}_{\mathcal{R}}$ are independent random matrices whose rows are i.i.d. Gaussian vectors, there exists a standard Gaussian ensemble $\mathbf{\widetilde X}_{\mathcal{R}}\in\re^{n\times p}$ such that $\mathbf{X}_{\mathcal{R}}=\mathbf{\widetilde X}_{\mathcal{R}}{\bfSigma}_S^{1/2}$. In the following, we set $\bfX:=\mathbf{\widetilde X}_{\mathcal{R}}$ for convenience. For each $(\bb,\btheta)\in\overline V_1(r_1)\times \overline V_2(r_2)$, we define
\begin{align}
W_{\bb,\btheta}&:=({\bfSigma}_S^{1/2} \bb)^\top\bfX^\top \btheta,\\
Z_{\bb,\btheta}&:=({\bfSigma}_S^{1/2} \bb)^\top\bfX^\top\overline{\btheta}+\overline \bb^\top\mathbf{\overline{X}}^\top\btheta,
\end{align}
where $\overline{\btheta}\in\mathbb{S}^{n-1}$ and $\overline{\bb}\in\mathbb{S}^{p-1}$ are fixed and $\mathbf{\overline X}$ is an independent copy of $\bfX$. Since $\bfX$ and $\mathbf{\overline X}$ are independent centered Gaussian ensembles, $(\bb,\btheta)\mapsto W_{\bb,\btheta}$ and $(\bb,\btheta)\mapsto Z_{\bb,\btheta}$ define centered continuous Gaussian processes $W$ and $Z$ indexed over $\overline V_1(r_1)\times \overline V_2(r_2)$. 

We shall now compute the increments of $W$. Setting $\tilde\bb:={\bfSigma}_S^{1/2}\bb$, we get
\begin{align}
W_{\tilde\bb,\btheta}-W_{\tilde \bb',\btheta'}=\sum_{i=1}^p-(\bfX_{\bullet,i})^\top(\tilde\bb_i\btheta-\tilde\bb_i'\btheta').
\end{align}
The above sum is the sum of $p$ independent random variables $\bfX_{\bullet,i}^\top(\tilde\bb_i\btheta-\tilde\bb_i'\btheta')\sim{\calN}(0,\Vert \tilde\bb_i\btheta-\tilde\bb_i'\btheta'\Vert_2^2)$. Hence,
\begin{align}
\esp\left[\left(\sum_{i=1}^p\bfX_{\bullet,i}^\top(\tilde\bb_i\btheta-\tilde\bb_i'\btheta')\right)^2\right]&=
\sum_{i=1}^p\Vert \tilde\bb_i\btheta-\tilde\bb_i'\btheta'\Vert_2^2=\Vert \tilde\bb\btheta^\top-\tilde \bb'(\btheta')^\top\Vert_{2,2}^2 \\
&=\Vert (\tilde\bb-\tilde \bb')\btheta^\top+\tilde \bb'(\btheta-\btheta')^\top\Vert_{2,2}^2 \\
&=\sum_{i=1}^p\sum_{j=1}^n[(\tilde\bb_i-\tilde \bb'_i)\btheta_j+\tilde \bb'_i(\btheta_j-\btheta'_j)]^2 \\
&=\Vert \tilde\bb-\tilde \bb'\Vert_2^2\Vert \btheta\Vert_2^2+\Vert \tilde \bb'\Vert_2^2\Vert \btheta-\btheta'\Vert_2^2-2(\Vert \tilde \bb'\Vert_2^2-\tilde \bb^\top \bb')(\Vert \btheta\Vert_2^2-\btheta^\top \btheta') \\
&\le \Vert \tilde\bb-\tilde \bb'\Vert_2^2+\Vert \btheta-\btheta'\Vert_2^2,
\label{lemma:aux2:eq1}
\end{align}
using Cauchy-Schwarz and the facts that $\Vert \btheta\Vert_2^2\le1$ and $\Vert \tilde \bb'\Vert_2^2\le1$.

We now consider the increments of $Z$. We have
\begin{align}
Z_{\tilde\bb,\btheta}-Z_{\tilde \bb',\btheta'}&=\sum_{i=1}^p\left(\tilde\bb_i-\tilde\bb_i'\right)\bfX_{\bullet,i}^\top\overline{\btheta}
+\sum_{i=1}^p\overline{\bb}_i\mathbf{\overline X}_{\bullet,i}^\top(\btheta-\btheta').
\end{align}
We now use the facts that $\{\bfX_{\bullet,i}\}_{i\in[p]}$ and $\{\mathbf{\overline X}_{\bullet,i}\}_{i\in[p]}$ are i.i.d. centered sequences independent of each other, the first sum is a sum of $p$ i.i.d. ${\calN}(0,|\tilde\bb_i-\tilde \bb'_i|^2)$ since $\Vert\overline{\btheta}\Vert_2=1$ while the second sum is a sum of $p$ i.i.d. ${\calN}(0,|\overline\bb_i|^2\Vert \btheta-\btheta'\Vert_2^2)$. From these facts and $\Vert\overline{\bb}\Vert_2=1$, we obtain that
\begin{align}
Z_{\tilde\bb,\btheta}-Z_{\tilde \bb',\btheta'}=\Vert \tilde\bb-\tilde \bb'\Vert_2^2+\Vert \btheta-\btheta'\Vert_2^2.
\label{lemma:aux2:eq2}
\end{align}

From \eqref{lemma:aux2:eq1}-\eqref{lemma:aux2:eq2}, we conclude that the centered Gaussian processes $W$ and $Z$ satisfy the conditions of Theorem \ref{thm:slepian}. Hence
\begin{align}
\esp\left[\sup_{(\bb,\btheta)\in\overline V_1(r_1)\times\overline V_2(r_2)}W_{\bb,\btheta}\right]&\le \esp\left[\sup_{(\bb,\btheta)\in\overline V_1(r_1)\times\overline V_2(r_2)}Z_{\bb,\btheta}\right]\\
&=\esp\left[\sup_{\bb\in\overline V_1(r_1)}({\bfSigma}_S^{1/2}\bb)^\top\bfX^\top\overline{\btheta}\right]+
\esp\left[\sup_{\btheta\in\overline V_2(r_2)}\overline{\bb}^\top\bfX^\top\btheta\right].
\end{align}
The first expectation can be bounded as 
\begin{align}
\esp\left[\sup_{\bb\in\overline V_1(r_1)}({\bfSigma}_S^{1/2}\bb)^\top\bfX^\top\overline{\btheta}\right]
\le\sup_{\bb\in\overline V_1(r_1)}\Vert \bb\Vert_1\cdot \esp\left[\Vert{\bfSigma}_S^{1/2}\bfX^\top\overline{\btheta}\Vert_\infty\right]
\le r_1\esp\left[\Vert{\bfSigma}_S^{1/2}\bfX^\top\overline{\btheta}\Vert_\infty\right].
\end{align}
Since $\bfX$ is a standard Gaussian ensemble and $\Vert\overline{\btheta}\Vert_2=1$, we have that, for all $i\in[p]$, $({\bfSigma}_S^{1/2}\bfX^\top\overline{\btheta})_i$ is a centered Gaussian random variable with variance $({\bfSigma}_S)_{ii}$. This fact and Theorem 2.5 in \citep{boucheron:lugosi:massart2013} imply that
$$
\esp\left[\Vert{\bfSigma}_S^{1/2}\bfX^\top\overline{\btheta}\Vert_\infty\right]\le \rho({\bfSigma}_S)\sqrt{2\log p},
$$ 
where $\varrho({\bfSigma}_S):=\max_{i\in[p]}({\bfSigma}_S)_{ii}$.

The second expectation can be bounded as 
\begin{align}
\esp\left[\sup_{\btheta\in\overline V_2(r_2)}\btheta^\top\bfX\overline{\bb}\right]
\le r_2\esp\left[\Vert\bfX\overline{\bb}\Vert_\infty\right] 
\le r_2\sqrt{2\log n},
\end{align}
by analogous reasons using that $\bfX\overline{\bb}\sim{\calN}_n(0,\bfI_n)$ which follows from the facts that $\mathbf{\bfX}\in\re^{n\times p}$ is a standard Gaussian ensemble and $\overline{\bb}\in\mathbb{S}^{p-1}$.

The claim of the lemma is proved by assembling \eqref{lemma:aux2:eq0} and the three previous inequalities and then normalizing by $\sqrt{n}$. 
\end{proof}

\end{document}